\documentclass{amsart}
\usepackage{amsmath}
\usepackage{amssymb}
\usepackage{amsthm}
\usepackage{dsfont}
\usepackage{lineno}
\usepackage{subfigure}
\usepackage{graphicx}
\usepackage{multirow}
\usepackage{color}
\usepackage{xcolor}
\usepackage{xspace}
\usepackage{pdfsync}
\usepackage{hyperref} 
\usepackage{algorithmicx}
\usepackage{algpseudocode}
\usepackage{algorithm}
\usepackage{mathtools}
\usepackage{enumitem}
\usepackage[normalem]{ulem} 

\graphicspath{{Figures/}}

\newcommand{\bq}{\begin{equation}}
\newcommand{\eq}{\end{equation}}

\algnewcommand{\LineComment}[1]{\State \(\triangleright\) #1}

\newtheorem{theorem}{Theorem}

\theoremstyle{lemma}
\newtheorem{lemma}[theorem]{Lemma}

\newtheorem{corollary}[theorem]{Corollary}

\newtheorem{definition}[theorem]{Definition}

\newtheorem*{remarkun}{Remark}

\theoremstyle{remark}

\makeatletter
\newcommand\appendix@section[1]{%
\refstepcounter{section}%
\orig@section*{Appendix \@Alph\c@section: #1}%
}
\let\orig@section\section
\g@addto@macro\appendix{\let\section\appendix@section}
\makeatother

\begin{document}

\title[A Volumetric Approach to Monge's Optimal Transport on Surfaces]{A Volumetric Approach to Monge's Optimal Transport on Surfaces}

\author{Richard Tsai}
\address{Department of Mathematics, University of Texas at Austin, Austin, TX, 78712}
\email{ytsai@math.utexas.edu}

\author{Axel G. R. Turnquist}
\address{Department of Mathematics, University of Texas at Austin, Austin, TX, 78712}
\email{agrt@utexas.edu}

\begin{abstract}
We propose a volumetric formulation for computing the Optimal Transport problem defined on surfaces in $\mathbb{R}^3$, found in disciplines like optics, computer graphics, and computational methodologies. Instead of directly tackling the original problem on the surface, we define a new  Optimal Transport problem on a thin tubular region, $T_{\epsilon}$, adjacent to the surface. This extension offers enhanced flexibility and simplicity for numerical discretization on Cartesian grids. 
The Optimal Transport mapping and potential function computed on $T_{\epsilon}$ are consistent with the original problem on surfaces. We demonstrate that, with the proposed volumetric approach, it is possible to use simple and straightforward numerical methods to solve  Optimal Transport for $\Gamma = \mathbb{S}^2$ and the $2$-torus.
\end{abstract}

\date{\today}    
\maketitle

\section{Introduction}\label{sec:introduction}

We consider computational Optimal Transport problems on smooth hypersurfaces in $\mathbb{R}^3$, with the metric induced by the Euclidean distance in $\mathbb{R}^3$.
Our objective is to derive an Optimal Transport problem and solve the corresponding Monge-Ampère-like Optimal Transport partial differential equations in a small neighborhood around the hypersurface. This volumetric approach brings more flexibility in dealing with surfaces and covariant derivatives.

In recent years, there has been much interest in finding solving Optimal Transport problems on Riemannian manifolds, mostly motivated from applications. These applications can roughly be divided into two categories, the first where the Optimal Transport problem is derived from first principles, and the second where Optimal Transport is used in an \textit{ad hoc} way as a powerful tool, usually in geometric and statistical analysis.

In freeform optics, a typical goal is to solve for the shape of reflectors or lenses that take a source light intensity to a desired target intensity pattern. The Optimal Transport partial differential equations arise as a consequence of Snell's law, the optical setup, and the conservation of light intensity. These formulations are inverse problems in which the potential function of Optimal Transport is directly related to the shapes of lenses or reflectors that redirect source light intensities to required target intensities, see~\cite{YadavThesis, RomijnThesis, RomijnSphere}, which include examples involving Optimal Transport PDE and other example whose formulations are Generated Jacobian Equations. In Section~\ref{sec:examples}, we will perform one example computation for the cost function arising in the reflector antenna problem, see~\cite{Wang_Reflector} and~\cite{Wang_Reflector2}, which is perhaps the most well-known example of such freeform optics problems with an Optimal Transport formulation.

A completely different application arises due to the fact that the solution to the Optimal Transport PDEs, called the potential function, is the Fr\'{e}chet gradient of the Wasserstein distance; see~\cite{santambrogio}. 
The potential function, therefore, has applications in non-parametric methods in inverse problems, where the Wasserstein distance is used as a ``fit" for data, see~\cite{gangbao}, for example.

In the adaptive mesh community, Optimal Transport has been used as a convenient tool for finding a mapping that redistributes mesh node density to a desired target density. The first such methods in the moving mesh methods community were proposed in~\cite{Weller_OTonSphere}. It is also used for the more general problem of diffeomorphic density matching on the sphere, where other approaches such as Optimal Information Transport can be used, see~\cite{modin}. 

In the statistics community, Optimal Transport has been extensively employed in computing the distance, or interpolating between, probability distributions on manifolds, and also in sampling, since solving the Optimal Transport problem allows one to compute a pushforward mapping for measures. In these situations, using Optimal Transport is not the only available computational tool, but has shown to be useful in these communities for its regularity and metric properties, even when source and target probability measures are not smooth or bounded away from zero. 


Many methods in the last ten or so years have been proposed for computing Optimal Transport related quantities on manifolds, such as the Wasserstein distance. Some of the computations were motivated by applications in computer graphics. In~\cite{solomon}, for example, the authors computed the Wasserstein-1 distance on manifolds in order to obtain the geodesic distance on complicated shapes, using a finite-element discretization. In~\cite{solomon2}, the authors used the Benamou-Brenier formulation of Optimal Transport, a continuous ``fluid-mechanics" formulation (as opposed to the static formulation presented here), to compute the Optimal Transport distance, interpolation, mapping, and potential functions on a triangulated approximation of the manifold. However, straightforward implementation using the Brenier-Benamou formulation may suffer from slow convergence. Various authors have successfully employed semi-discrete Optimal Transport methods to Optimal Transport computations on surfaces, where the source measure is a sum of Dirac delta measures and the target measure is continuous. The resulting mapping of such semi-discrete Optimal Transport problems is related to a weighted power diagram of the locations of the source Dirac delta measures. In~\cite{paraboloids}, the authors constructed power diagrams on the sphere to solve the reflector antenna problem, with cost function $c(x,y) = -\log(1-x \cdot y)$. In~\cite{sphericalOT}, the authors utilized a semi-discrete Optimal Transport computation to redistribute the local area of a mesh on the sphere, using the cost function $c(x,y) = \log(x \cdot y)$. In~\cite{simplexsoup}, the authors designed an algorithm to compute the distance between surfaces in $\mathbb{R}^3$ and point clouds, using the cost function $c(x,y) = \left\Vert x - y \right\Vert^2$, where $\left\Vert \cdot \right\Vert$ denotes the Euclidean distance.  

We now briefly review methods for solving Optimal Transport problems on surfaces.
One common approach to developing numerical methods for surface problems starts with triangularization of the manifold. Recently, in the paper~\cite{osher_mfg}, mean-field games were discretized and solved on manifolds using triangularization, of which the Benamou-Brenier formulation of Optimal Transport (using the squared geodesic cost function) is a subcase. One of the great challenges of applying traditional finite element methods to Optimal Transport PDE is the fact that they are not in divergence form. Thus, the analysis becomes very challenging, but, nevertheless, considerable work has been done for this in the Euclidean case with the Monge-Amp\`{e}re PDE, see, for example, the work done in~\cite{Brenner}. It is also not simple to design higher-order schemes, in contrast to simple finite-differences on a Cartesian grid where it is very simple to design high-order discretizations.

There is another general approach to approximating PDE on manifolds, which is done by locally approximating both the manifold and the function to be computed via polynomials using a moving least squares method, originally proposed in~\cite{zhao}. These polynomial approximations are computed in a computational neighborhood. Thus far, these methods have not been applied to solving Optimal Transport PDE on the sphere.


The closest point method was originally proposed in~\cite{ruuth} (see also \cite{macdonald2010}), where the solution of certain PDE are extended to be constant via a closest point map in a small neighborhood of the manifold and all derivatives are then computed using finite differences on a Cartesian grid. However, for our purposes, using the closest point operator for the determinant of a Hessian is rather complicated. Furthermore, simply extending the source and target density functions will not necessarily lead to an Optimal Transport problem on the extended domain and extending the cost function via a closest point extension without introducing a penalty in the normal directions will lead to a degenerate PDE. We have decided instead to extend the source and target density functions and compensating using Jacobian term so they remain density functions and extend the cost function with a penalty term. As we will show in Section~\ref{sec:extension}, these choices will naturally lead to a solution which is constant with respect to the closest point map. Also, we will find in Section~\ref{sec:extension} that our re-formulation leads to a natural new Optimal Transport problem in a tubular neighborhood with natural zero-Neumann boundary conditions which may be of independent interest for the Optimal Transport community.


In this article, instead of solving directly the Monge problem of Optimal Transport defined on $\Gamma$, we propose solving the equivalent Optimal Transport problem on a narrowband $T_{\epsilon}$ around $\Gamma$
for a special class of probability measures in $T_\epsilon$ that is constructed from probability measures on $\Gamma$, and with a class of cost functions that is derived from cost functions on $\Gamma$.
Similar to \cite{martin2020equivalent}, 
 our approach in this paper is to reformulate the variational formulation of the problem on the manifold, which is done in Section~\ref{sec:extension}. 
 It relies on the fact that the extension presented in Section~\ref{sec:extension} is itself an Optimal Transport problem, and so the usual techniques for the PDE formulation of Optimal Transport are used to formulate the PDE on the tubular neighborhood $T_{\epsilon}$. 
We will demonstrate, however, that solving this new Optimal Transport problem in $T_\epsilon$ will not require that we take thickness of the narrowband to zero on the theoretical level. It is a good idea in computation, however, to take the thickness of the narrowband to zero so that the computational complexity of the scheme scales well.
This new Optimal Transport problem is set up carefully with a cost function that penalizes mass transport in the normal direction to the manifold $\Gamma$. 
Because the method in this manuscript allows for great flexibility in the choice of cost function, it can also be employed to solve the reflector antenna problem, which involves finding the shape of a reflector given a source directional light intensity and a desired target directional light intensity. Some methods, such as those developed for the reflector antenna and rely on a direct discretization of the Optimal Transport PDE, include~\cite{Oliker_nearfield, Oliker_SQM, Doskolovich_farfield, GlimmOliker_SingleReflector, Wu_lensdesign, Brix_MAOptics, RomijnSphere, HT_OTonSphere3}. However, these methods (with the exception of~\cite{HT_OTonSphere3}) have been designed solely for the reflector antenna problem; that is, they are restricted to the cost function $c(\mathbf{x},\mathbf{y}) = -\log \left\Vert \mathbf{x} - \mathbf{y} \right\Vert$.

The PDE method proposed in this manuscript can be contrasted with the wide-stencil monotone scheme, that has convergence guarantees, developed in~\cite{HT_OTonSphere, HT_OTonSphere2}, where discretization of the second-directional derivatives was performed on local tangent planes. The Cartesian grid proposed here is much simpler, which makes the discretization of the derivatives much simpler. The greatest benefits of the current proposed scheme over monotone methods are that a wider variety of difficult computational examples are possible to compute in a short amount of time with accelerated convergence techniques, see Algorithm~\ref{alg:solver} which shows that the current implementation uses an accelerated single-step Jacobi method. A possible slowdown that might be expected from extending the discretization to a third dimension is counteracted by the efficiency of the discretization (which does not require computing a large number of derivatives in a search radius, as is done in the monotone scheme in~\cite{HT_OTonSphere2}) and the good performance of the accelerated solvers for more difficult computational examples that allow for much larger time step sizes in practice in the solvers than are typically used in monotone schemes.

It has been shown in~\cite{HL_ThreeDimensions} how to construct wide-stencil monotone finite-difference discretizations, for regions in $\mathbb{R}^3$, of the Optimal Transport PDE on $T_{\epsilon}$. However, we believe that the convergence guarantees will be outweighed by the computational challenges of requiring a large number of discretization points to resolve $T_{\epsilon}$ especially for relatively small $\epsilon$. Also, in~\cite{HL_ThreeDimensions}, it was shown, for technical reasons, that discretization points had to remain a certain distance away from the boundary. We argue that in a region like $T_{\epsilon}$ where most points are close to the boundary this is too restrictive of a choice. The work in~\cite{HL_ThreeDimensions} was, however, for more general Optimal Transport problems on regions in $\mathbb{R}^3$ and it may be possible to construct monotone discretizations in $T_{\epsilon}$ by exploiting the symmetry of our Optimal Transport problem in $T_{\epsilon}$, due to the fact that the solution is constant in the normal direction, but we defer a detailed discussion of this point to future work.

The method proposed in this manuscript is a direct discretization of the full Optimal Transport PDE using a grid that is generated from a Cartesian cube of evenly spaced points, with computational stencils formed from the nearest surrounding points. This can be contrasted with the wide-stencil schemes in~\cite{HT_OTonSphere2} and the geometric methods in~\cite{Weller_OTonSphere}, one of the earliest methods proposed for solving the Optimal Transport problem on the sphere with squared geodesic cost. Although the computational methods in~\cite{Weller_OTonSphere} perform well, many properties of the discretizations were informed by trial and error.

In much of the applied Optimal Transport literature, the fastest method known for computing an approximation of the Optimal Transport distance is achieved by entropically regularizing the Optimal Transport problem and then using the Sinkhorn algorithm, originally proposed in~\cite{cuturi}. If one wishes to compute an approximation of the distance between two probability measures, then Sinkhorn is the current state of the art. However, it is unclear from the transference plane one obtains from the entropically regularized problem how to extract the Optimal Transport mapping and the potential function. Nevertheless, one can entropically regularize our extended Optimal Transport problem on $T_{\epsilon}$ and run the Sinkhorn algorithm to efficiently compute an approximation of the Wasserstein distance between two probability distributions. For our proposed extension, the Wasserstein distance (between two probability distributions) for the Optimal Transport problem on $\Gamma$  will be equal to the Wasserstein distance between the extended probability distributions on $T_{\epsilon}$. Moreover, using the divergence as a stopping criterion defined in~\cite{cuturi}, after a brief investigation, we found that the Sinkhorn algorithm requires approximately the same number of iterations to reach a given tolerance (of the divergence) for the Optimal Transport problem on $\Gamma$ as for the Optimal Transport problem on $T_{\epsilon}$.



In Section~\ref{sec:background}, we review the relevant background for the PDE formulation of Optimal Transport on manifolds. We then introduce the Monge problem of Optimal Transport and then characterize the minimizer as a mass-preserving map that arises from a $c$-convex potential function. In Section~\ref{sec:extension}, we set up the new Optimal Transport problem carefully and prove how the map from of the original Optimal Transport problem on $\Gamma$ can be extracted from the map from the new Optimal Transport problem on $T_{\epsilon}$. The PDE formulation of the Optimal Transport problem on $T_{\epsilon}$ also naturally comes with Neumann boundary conditions on $T_{\epsilon}$. In Section~\ref{sec:examples}, we detail how we construct a discretization of the Optimal Transport PDE on $T_{\epsilon}$ by using a Cartesian grid, and then show how we solve the resulting system of nonlinear equations. Instead of discretizing the zero-Neumann conditions, we will instead use the fact that the solution of the PDE on $T_{\epsilon}$ is constant in the direction normal to $\Gamma$ to extrapolate. Then, we run sample computations on the sphere and on the $2$-torus for different cost functions and run some tests where we vary $\epsilon$ and our discretization parameter $h$. In Section~\ref{sec:conclusion}, we recap how the reformulation has allowed us to design a simple discretization of the Optimal Transport PDE.

\section{Background}\label{sec:background}
Throughout this paper, unless otherwise stated, we assume that $\Gamma\subset\mathbb{R}^3$ is a closed, smooth, and simply connected surface. Let $dS(\mathbf{x})$ denote the surface area element at $\mathbf{x}\in\Gamma$, induced from $\mathbb{R}^3$.
We consider probability measures of the form
\( d\mu(\mathbf{x}) = \rho(\mathbf{x})dS(\mathbf{x}) \). 
The mapping
\( \boldsymbol{\xi}:\Gamma \rightarrow \Gamma \)
is said to push forward the probability measure $\mu$ onto $\nu$ if $\nu=\mu\circ \boldsymbol{\xi}^{-1}$.
This action is denoted by \(\boldsymbol{\xi}_{\#} \mu = \nu \).


Given ``source" and ``target" probability measures, $\mu$ and $\nu$, and a cost function $c: \Gamma \times \Gamma \rightarrow \mathbb{R}$, the Monge problem on $\Gamma$ is to find a map $\boldsymbol{\xi}$, 
satisfying $\boldsymbol{\xi}_{\#}\mu = \nu$, that minimizes the cost functional
\begin{equation}\label{eq:orig_cost_fncl}
\mathcal{C}(\boldsymbol{\xi}) := \int_{\Gamma} c(\mathbf{x},\boldsymbol{\xi}(\mathbf{x})) d\mu(\mathbf{x}).
\end{equation}
The existence of a minimizer for this optimization problem is guaranteed when 
the probability measures admit densities on $\Gamma$ and $c$ is lower-semicontinuous; see~\cite{santambrogio}. 

We will consider probability measures admitting smooth density functions bounded away from zero:
\begin{equation}\label{eq:Theta_Gamma}
    \Theta_{\Gamma} := \left\{ \rho(\mathbf{x}) \in C^{\infty}: \int_{\Gamma}\rho dS = 1, \ \rho(\mathbf{x})>0, \forall \mathbf{x} \in \Gamma \right\}.
\end{equation}

Assuming that the cost function $c$ satisfies the so-called MTW conditions~\cite{MTW}, Monge's problem on $\Gamma$ has a unique smooth solution $\mathbf{m}$, called the Optimal Transport mapping:
\begin{equation}\label{eq:OT}
\mathbf{m} = \text{argmin}_{\boldsymbol{\xi}_{\#}\mu = \nu} 
\mathcal{C}(\boldsymbol{\xi}).
\end{equation}

In~\cite{mccann}, it was shown that when the cost function is the squared geodesic cost $c(x,y) = \frac{1}{2}d_{\Gamma}(x,y)^2$, the Optimal Transport mapping takes the simple form
\begin{equation}
    \mathbf{m}(\mathbf{x})=\exp_\mathbf{x}(\nabla u(\mathbf{x})),
\end{equation}
where $u$ is a special scalar function, called the potential function, and $\exp_\mathbf{x}(\nabla u(\mathbf{x}))$, the exponential map,  denotes the point on $\Gamma$
connected to the geodesic emanating from $\mathbf{x}$ in the direction of $\nabla u(\mathbf{x})$ for a length equal to $|\nabla u(\mathbf{x})|$ (assuming $\Gamma$ is such that such a geodesic exists and is unique). 

The uniqueness of minimizer $\mathbf{m}$ of Equation~\eqref{eq:OT} can be characterized by the Optimall Transport problem's potential function being \emph{c-convex}, with $c$ associated with the cost function of the problem.
\begin{definition}
The $c$-transform of a function $u:\Gamma \rightarrow \mathbb{R}$, which is denoted by $u^{c}$ is defined as:
\begin{equation}
    u^{c}(\mathbf{y}) = \sup_{\mathbf{x} \in \Gamma} \left( -c(\mathbf{x},\mathbf{y}) - u(\mathbf{x}) \right).
\end{equation}
\end{definition}
\begin{definition}\label{def:cconvex}
A function $u$ is $c$-convex if at each point $\mathbf{x} \in \Gamma$, there exists a $\mathbf{y} \in \Gamma$ such that 
\begin{equation}
\begin{cases}
-u^{c}(\mathbf{y}) - c(\mathbf{x},\mathbf{y}) = u(\mathbf{x}), \\
-u^{c}(\mathbf{y}) - c(\mathbf{x}', \mathbf{y}) \leq u(\mathbf{x}'), \ \forall \mathbf{x}' \in \Gamma,
\end{cases}
\end{equation}
where $u^{c}(\mathbf{y})$ is the $c$-transform of $u$.
\end{definition}



Let $u\in C^1(\Gamma;\mathbb{R})$ and $c$-convex, we define implicitly a mapping $\tilde{\mathbf{m}}$ as the solution to the equation:
\begin{equation}\label{eq:relation}
\nabla_{\Gamma} u (\mathbf{x})=-\nabla_{\mathbf{x}, \Gamma} c(\mathbf{x},\tilde{\mathbf{m}}),~~~\forall\mathbf{x}\in\Gamma,
\end{equation}
where the gradients $\nabla_{\Gamma}$ are taken with respect to the metric on $\Gamma$, see~\cite{LoeperReg}. 
If such a mapping $\tilde{\mathbf{m}}$ also satisfies $\tilde{\mathbf{m}}_{\#}\mu = \nu$, then it is exactly the unique solution of the Optimal Transport problem in Equation~\eqref{eq:OT}.

The preceding discussion is summarized in Theorem 2.7 from~\cite{LoeperReg}:
\begin{theorem}\label{thm:uniqueness}
The Monge problem in Equation~\eqref{eq:OT} with smooth cost function $c(\mathbf{x},\mathbf{y})$ satisfying the MTW conditions (see~\cite{MTW}) and source and target probability measures $\mu$ and $\nu$, respectively, where $\mu$ and $\nu$ have density functions $f,g\in \Theta_{\Gamma}$, respectively, has a solution which is a mapping $\mathbf{m}$ iff $\mathbf{m}$ satisfies both $\mathbf{m}_{\#} \mu = \nu$ and is uniquely solvable via Equation~\eqref{eq:relation}, where $u$ is a $c$-convex function.
\end{theorem}

\begin{remarkun}
Such a $c$-convex potential function $u$ will be unique up to an additive constant. The gradient $\nabla_{\Gamma}u$, however, is unique.
\end{remarkun}

Assuming the MTW conditions for the cost function $c$, and source and target probability measures in $\Theta_\Gamma$,  $u$ and $\mathbf{m}$ are smooth classical solutions of the following equations:
\begin{equation}\label{eq:OTGamma1}
    \det(D_{\Gamma}^2 u(\mathbf{x}) + D^2_{\mathbf{x}\mathbf{x}, \Gamma}c(\mathbf{x},\mathbf{y})) = \left\vert \det D^2_{\mathbf{x}\mathbf{y}, \Gamma} c(\mathbf{x},\mathbf{m}(\mathbf{x})) \right\vert f(\mathbf{x})/ g(\mathbf{m}(\mathbf{x})),
\end{equation}
\begin{equation}\label{eq:OTGamma2}
    \nabla_{\Gamma} u(\mathbf{x}) = -\nabla_{\mathbf{x}, \Gamma} c(\mathbf{x},\mathbf{m}(\mathbf{x})).
\end{equation}
Here the derivatives $D_{\Gamma}$ are taken on the surface with respect to the induced metric on $\Gamma$.
We will refer to \eqref{eq:OTGamma1}-\eqref{eq:OTGamma2} as the Optimal Transport PDEs.

In the next Section, we will derive a special extension of the optimal transport problem \eqref{eq:OT} defined on $\Gamma$, for densities in \eqref{eq:Theta_Gamma}, to one defined in a tubular neighborhood (narrowband) $T_{\epsilon}$ of $\Gamma$ for a class of probability measures judiciously extended from
\eqref{eq:Theta_Gamma}. The plan is to solve, in $T_\epsilon$, the Optimal Transport PDEs \eqref{eq:OTExtension}-\eqref{eq:bc}, corresponding to the extended optimal transport problem. The PDEs in $T_{\epsilon}$ will be defined with
partial derivatives of the potential functions with respect to the standard basis in the ambient Euclidean space.
This approach will allow one to work with surface problems without the need to work with surface parameterization and to enjoy the benefit of choosing a mature numerical method for discretizing the Monge-Ampère-like Optimal Transport PDEs.

\section{Volumetric Extension of the Optimal Transport Problem}\label{sec:extension}
In this section, we define a special Optimal Transport problem on a tubular neighborhood of $\Gamma$ which is equivalent in a suitable sense to the 
given Optimal Transport problem defined on a closed smooth surface $\Gamma \subset \mathbb{R}^3$.
The Optimal Transport mapping from this extended Optimal Transport problem will be shown to have a natural connection with the Optimal Transport mapping on $\Gamma$. 
Working in a tubular neighborhood of the surface allows for flexible meshing and use of standard discretization methods for the differential and integral operators in the extended Euclidean domain.
The proposed approach follows the strategy developed in \cite{tsai} and \cite{martin2020equivalent}. 

We formulate the Optimal Transport mapping for the Monge Optimal Transport problem on $T_{\epsilon}$ as 
\begin{equation}\label{eq:OT2}
\overline{\mathbf{m}} = \text{argmin}_{\boldsymbol{\xi}_{\#}\overline{\mu} = \overline{\nu}} \int_{T_{\epsilon}} c_{\sigma}\left(\mathbf{z}',\boldsymbol{\xi}(\mathbf{z}') \right) d \overline{\mu}(\mathbf{z}'),
\end{equation}
for a special class of cost function $c_{\sigma}: T_\epsilon\times T_\epsilon\mapsto \mathbb{R}^+$ and
a special class of probabilities $\overline{\mu}$ and $\overline{\nu}$ with density functions $\overline{f}$ and $\overline{g}$. 

Given $f, g \in \Theta_{\Gamma}$ that are source and target densities and $c$ a cost function on $\Gamma$ in the original Optimal transport problem, we will present a particular way of extending $f, g$ and $c$ to $\overline{f}, \overline{g} \in \Theta_{T_{\epsilon}}$ and $c_{\sigma}$ in Section~\ref{sec:densities} and Section~\ref{sec:costs}, respectively.
In Section~\ref{sec:maintheorem}, we will show that the Optimal Transport problem in Equation~\eqref{eq:OT2} is ``equivalent" to Equation~\eqref{eq:OT} in a specific sense which will be made clear in Theorem~\ref{thm:maintheorem}.



With the judiciously extended $\overline{f}$, $\overline{g}$ and $c_{\sigma}$, our goal is to solve numerically (up to a constant) the following PDE for the pair $(v, \overline{\mathbf{m}})$: 

\begin{equation}\label{eq:OTExtension}
\begin{aligned}
 \det\left( D^2 v(\mathbf{z}) + D^2_{\mathbf{z}\mathbf{z}} c_{\sigma}(\mathbf{z}, \boldsymbol{\xi}) \right) &= \left\vert \det D^2_{\mathbf{z} \boldsymbol{\xi}} c_{\sigma}(\mathbf{z},\boldsymbol{\xi})\right\vert \overline{f}(\mathbf{z}) / \overline{g} (\boldsymbol{\xi}), \\
 \nabla v(\mathbf{z}) &= -\nabla_{\mathbf{z}} c_{\sigma}(\mathbf{z}, \boldsymbol{\xi}),\text{ for }{\boldsymbol{\xi} = \overline{\mathbf{m}}(\mathbf{z})}, \mathbf{z} \in T_{\epsilon},
 \end{aligned}
\end{equation}
with the Neumann boundary condition
\begin{equation}\label{eq:bc}
    \frac{\partial v(\mathbf{z})}{\partial \mathbf{n}} = 0, \ \ \ \mathbf{z} \in \partial T_{\epsilon}.
\end{equation}

In fact, we will find in Theorem~\ref{sec:maintheorem} that the potential function $v(\mathbf{z})$ satisfies
\begin{equation}
    \frac{\partial v(\mathbf{z})}{\partial \mathbf{n}} = 0, \ \ \ \mathbf{z} \in T_{\epsilon},
\end{equation}
that is, potential function is constant in the normal direction. We will exploit in our numerical solver, although we emphasize that one could impose the Neumann boundary condition in Equation~\eqref{eq:bc} in a numerical implementation.

We remark that all Optimal Transport problems posed on bounded subsets of Euclidean space have a natural global condition, known as the second boundary value problem, see~\cite{Urbas}. The second boundary condition is a global constraint that can be formulated as a global Neumann-type condition; see~\cite{HamfeldtBVP2}, which can be shown to reduce to Equation~\eqref{eq:bc} on the boundary.

\subsection{The general setup}\label{sec:geometricsetup}

 
With a predetermined orientation, we define the signed distance function to $\Gamma$:
\begin{equation}\label{eq:signed-distance-to-Gamma}
    y = \phi_\Gamma(\mathbf{z}):= \text{sgn}(\mathbf{z})\min_{\mathbf{x}\in\Gamma} ||\mathbf{z}-\mathbf{x}||,
\end{equation}
where $\text{sgn}$ corresponds to the orientation. Typically, one choose $\text{sgn}(\mathbf{z})<0$  for $\mathbf{z}$ in the bounded region enclosed by $\Gamma$. 
We denote the $y$-level set of $\phi_{\Gamma}$ by $\Gamma_\eta$; i.e.,
$$\Gamma_{y} = \left\{\mathbf{z} \in \mathbb{R}^3: \phi_{\Gamma}(\mathbf{z}) = y \right\}.$$

Let $\mathcal{C}_\Gamma$ denote the set of points in $\mathbb{R}^3$ which are equidistant to at least two distinct points on $\Gamma$. The reach of $\Gamma$ is an extrinsic property defined as
\begin{equation}\label{eq:reach}
\tau_\Gamma := \inf_{{\mathbf{x}}\in \Gamma,\, {\mathbf{z}}\in\mathcal{C}_\Gamma} ||{\mathbf{x}}-\mathbf{z}||.
\end{equation}
Note that by definition, $\tau_\Gamma \le 1/{\kappa}$ where $\kappa$ is the absolute value of the largest eigenvalue of the second fundamental form over all points in $\Gamma$. Furthermore, if $\Gamma$ is a $C^2$ surface in $\mathbb{R}^3$, its reach is bounded below from 0, see \cite{Federer}. The reach $\tau_{\Gamma}$ 
can be estimated from the distance function $\phi_\Gamma$, in particular from information on the skeleton, also known as the medial axis, of $\Gamma$; see e.g. \cite{kimmel1995skeletonization}\cite{Aamari/19-EJS1551}.

We will work with the tubular neighborhood $T_{\epsilon} \subset \mathbb{R}^3$ of $\Gamma$:
\begin{equation}\label{def:T_epsilon}
T_{\epsilon} := \left\{ \mathbf{z} \in \mathbb{R}^3: \left\Vert \mathbf{z}- \mathbf{p} \right\Vert < \epsilon < \tau_\Gamma, \mathbf{p} \in \Gamma \right\}.
\end{equation}
See Figure~\ref{fig:neighborhood}, for a schematic depiction.
In $T_\epsilon$, the closest point mapping to $\Gamma$
\begin{equation}\label{eq:closestpoint}
    P_{\Gamma}(\mathbf{z}) := \text{argmin}_{\boldsymbol{x} \in \Gamma} ||\mathbf{z}-\mathbf{x}||,
\end{equation}
and the normal vector field
 $
 \mathbf{n}(\mathbf{z}) :=\nabla\phi_\Gamma(\mathbf{z})
 $
 are well-defined; furthermore,  
 $$\mathbf{n}(\mathbf{z})=\mathbf{n}(P_\Gamma\mathbf{z})~~\text{and}~~\phi_\Gamma(\mathbf{z})=\mathbf{n}(\mathbf{z}) \cdot (\mathbf{z} - P_\Gamma{\mathbf{z}}).$$

\begin{figure}[htp]
	\centering
	\includegraphics[width=0.5\textwidth]{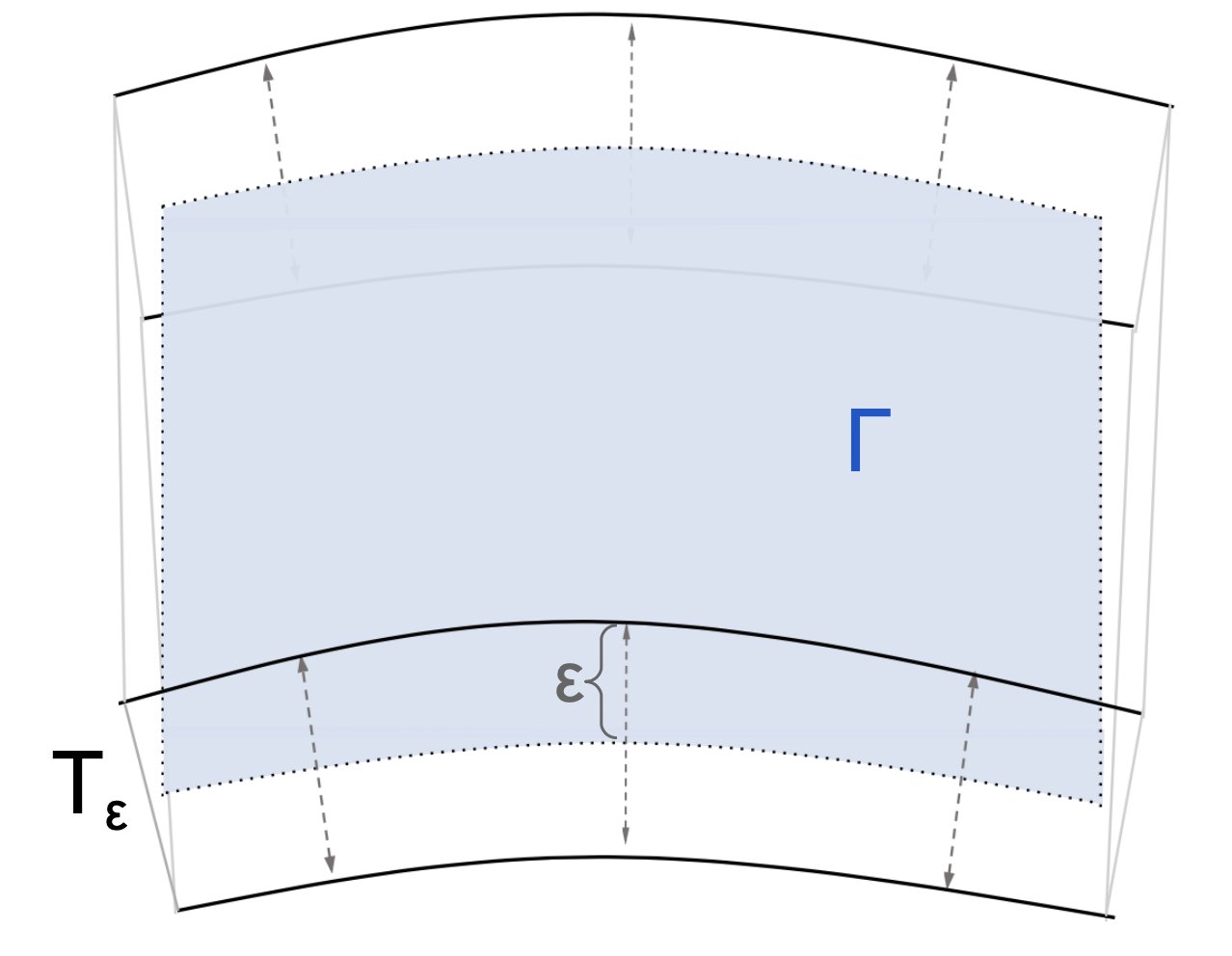}
	\caption{ New Optimal Transport problems  are defined in $T_\epsilon$ to have solutions to the Optimal Transport problems on $\Gamma$.
 }\label{fig:neighborhood}
\end{figure}
 
\subsection{Extension of Surface Densities}\label{sec:densities}

Let $\rho\in \Theta_{\Gamma}$, we can rewrite the integration of $f$ on $\Gamma$ to one on $\Gamma_\eta\subset T_\epsilon$ 
as follows:
\begin{equation}\label{eq:int_Gamma=int_Gamma_y}
    \int_{\Gamma}\rho(\mathbf{x})dS = \int_{\Gamma_{\eta}}\rho(P_{\Gamma}\mathbf{z}')J(\mathbf{z}')dS,
\end{equation}
where $J(\mathbf{z}')$
accounts for the change of variables in the integrations. In fact,
\begin{equation}\label{def:J}
  J(\mathbf{z}')=(1-\kappa_1 \phi_\Gamma)(1-\kappa_2 \phi_\Gamma)(\mathbf{z}')=\sigma_1(\mathbf{z}') \sigma_2(\mathbf{z}'),  
\end{equation}
where $\kappa_1,\kappa_2$ are the two  principal curvatures of $\Gamma_\eta$ at $\mathbf{z}^\prime$ whose signs are determined by the surface orientation,  and
$\sigma_1, \sigma_2$ are the two largest singular values of the Jacobian matrix of $ P_{\Gamma}$.
See e.g. \cite{kublik2016integration}.

Furthermore, by the coarea formula,
\begin{align*}
    1 = \int_{\Gamma}f(\mathbf{x})dS = \frac{1}{2\epsilon} \int_{-\epsilon}^\epsilon\int_{\Gamma_{y}}f(P_{\Gamma}\mathbf{z}')J(\mathbf{z}')dSd\eta
    &= \frac{1}{2\epsilon}\int_{T_\epsilon} f(P_{\Gamma}\mathbf{z}) J(\mathbf{z}) d\mathbf{z}.
\end{align*}

Therefore, our extended Optimal Transport problem is formulated with the class of density function defined below:
\begin{equation}\label{eq:densities}
    \Theta_{T_\epsilon}:= \left\{\frac{1}{2\epsilon}\rho(P_\Gamma\mathbf{z})J(\mathbf{z}): \rho\in\Theta_\Gamma \right\}.
\end{equation}

Therefore, any density in $\Theta_{T_\epsilon}$ is strictly positive and smooth (The Jacobian $J$ is smooth and strictly positive since we stay within the reach of $\Gamma$.)


\subsection{Extension of Cost Function to $T_\epsilon$}\label{sec:costs}

Let $c$ be the cost function defined on $\Gamma$. 
We define the extended cost function for any two points in $T_\epsilon$ by adding an additional cost to the difference in their distance to $\Gamma$: 

 \begin{equation}\label{extended-cost}
     c_{\sigma}(\mathbf{z}_1,\mathbf{z}_2) = c(P_\Gamma  \mathbf{z}_1, P_\Gamma  \mathbf{z}_2) + \frac{\sigma}{2} (\phi_\Gamma(\mathbf{z}_1)- \phi_\Gamma(\mathbf{z}_2))^2,~~~\text{for some}~\sigma>0.
 \end{equation}
We will see that the particular choice of $\sigma$ will not affect the analysis.

\subsection{Equivalence of the Two Optimal Transport Problems}\label{sec:maintheorem}

In this section, we establish that the Optimal Transport problem defined via our extensions of the density functions and the cost function in Sections~\ref{sec:densities} and~\ref{sec:costs} leads to an Optimal Transport mapping $\overline{\mathbf{m}}$ that moves mass only along each level set of the distance function to $\Gamma$. Moreover, $\overline{\mathbf{m}}$ can be used to find $\mathbf{m}$, the mapping from the Optimal Transport problem on $\Gamma$. 

\begin{theorem}\label{thm:maintheorem}
The solution, $\overline{\mathbf{m}}$, to the new Optimal Transport problem presented in Equation~\eqref{eq:OT2} defined with the densities in $\Theta_{T_\epsilon}$ and cost function in Equation~\eqref{extended-cost}, satisfies
\begin{equation}\label{eq:Pm_eps=m(P)}
    P_{\Gamma} \overline{\mathbf{m}}(\mathbf{z}) = \mathbf{m}(P_{\Gamma}\mathbf{z}),~~~\forall \mathbf{z}\in T_\epsilon,
\end{equation}
and
\begin{equation}\label{eq:phi_z=phim_z}
    \phi_\Gamma(\mathbf{z}) = \phi_\Gamma(\overline{\mathbf{m}}(\mathbf{z})),~~~\forall \mathbf{z}\in T_\epsilon.
\end{equation}
\end{theorem}

This theorem then implies that the Optimal Transport cost for the Optimal Transport problem on $\Gamma$ is equal to the Optimal Transport cost for the extended Optimal Transport problem on $T_{\epsilon}$. Thus, the Wasserstein distance, for example, on $\Gamma$ can be computed via the extended Optimal Transport problem on $T_{\epsilon}$.
\begin{corollary}
We have
\begin{equation}
\min_{\boldsymbol{\xi}_{\#}\overline{\mu} = \overline{\nu}} \int_{T_{\epsilon}} c_{\sigma}(\mathbf{z}, \boldsymbol{\xi}(\mathbf{z})) \overline{f}(\mathbf{z})d\mathbf{z} = \min_{\boldsymbol{\theta}_{\#}\mu = \nu} \int_{\Gamma}c(\mathbf{x}, \boldsymbol{\theta}(\mathbf{x})) f(\mathbf{x})dS(\mathbf{x}).
\end{equation}
\end{corollary}
\begin{proof}

\begin{align*}
\min_{\boldsymbol{\xi}_{\#}\overline{\mu} = \overline{\nu}} \int_{T_{\epsilon}} c_{\sigma}(\mathbf{z}, \boldsymbol{\xi}(\mathbf{z})) \overline{f}(\mathbf{z})d\mathbf{z} &= \int_{T_{\epsilon}} c_{\sigma}(\mathbf{z}, \overline{\mathbf{m}}(\mathbf{z})) \overline{f}(\mathbf{z})d\mathbf{z} \\
&=\int_{T_{\epsilon}} c(P_{\Gamma}\mathbf{z},P_{\Gamma}\overline{\mathbf{m}}(\mathbf{z}))\overline{f}(\mathbf{z})d\mathbf{z} \\
&=\frac{1}{2 \epsilon} \int_{-\epsilon}^{\epsilon} \int_{\Gamma_\eta} c(P_{\Gamma}\mathbf{z}, P_{\Gamma}\overline{\mathbf{m}}(\mathbf{z})) f(\mathbf{x})J(\mathbf{z})dS(\mathbf{z})d\eta \\
&=\frac{1}{2 \epsilon}\int_{-\epsilon}^{\epsilon} \int_{\Gamma}c(\mathbf{x}, \mathbf{m}(\mathbf{x})) f(\mathbf{x})dS(\mathbf{x})d\eta \\
&= \int_{\Gamma}c(\mathbf{x}, \mathbf{m}(\mathbf{x})) f(\mathbf{x})dS(\mathbf{x}) \\
&= \min_{\boldsymbol{\theta}_{\#}\mu = \nu} \int_{\Gamma}c(\mathbf{x}, \boldsymbol{\theta}(\mathbf{x})) f(\mathbf{x})dS(\mathbf{x}).
\end{align*}

\end{proof}

\subsubsection*{Mass preservation}

Since $\epsilon<\tau_\Gamma$ and $\Gamma$ is compact and closed, the projection $P_\Gamma$ is bijective between $\Gamma$ and $\Gamma_\eta$ for any $\eta\in(-\epsilon,\epsilon).$
We define the inverse map 
\begin{equation}
    P_{\Gamma,\Gamma_\eta}^{-1}(\mathbf{z}):=\mathbf{z}+\eta\mathbf{n}(\mathbf{z}),~~~\text{for}~\mathbf{z}\in\Gamma,
\end{equation}
where $\mathbf{n}(\mathbf{z})=\nabla\phi_\Gamma(\mathbf{z})$ is the outward normal vector of $\Gamma$ at $\mathbf{z}$. 
This vector remains the same for all $\mathbf{z}'\in T_\epsilon$ satisfying $P_\Gamma\mathbf{z}'=\mathbf{z}.$ Thus we see that $P_\Gamma\circ P^{-1}_{\Gamma,\Gamma_\eta}(\mathbf{z})=\mathbf{z}$ and $P^{-1}_{\Gamma,\Gamma_\eta}\circ P_\Gamma(\mathbf{z}')=\mathbf{z}'$ for $\mathbf{z}'\in\Gamma_\eta.$
Therefore, we can associate any mapping $\boldsymbol{\xi}:\Gamma\mapsto\Gamma$ with $\boldsymbol{\xi}\circ P_\Gamma:\Gamma_\eta\mapsto\Gamma_\eta$ and vice versa.
So we define
\begin{equation}\label{eq:extension_by_restriction}
    \boldsymbol{\xi}_\epsilon(\mathbf{z}) := P_{\Gamma,\Gamma_\eta}^{-1}\circ\boldsymbol{\xi}\circ P_\Gamma(\mathbf{z}),~~~\phi_\Gamma(\mathbf{z})=\eta,~~\text{and}~\eta\in(-\epsilon,\epsilon).
\end{equation}
Equivalently, 
\begin{equation}\label{eq:extension_by_restriction2} P_{\Gamma}\boldsymbol{\xi}_\epsilon(\mathbf{z}')=\boldsymbol{\xi}(P_\Gamma \mathbf{z}'),~~~\mathbf{z}\in T_\epsilon,
\end{equation}
and
\begin{equation}\label{eq:extension_by_restriction3} \phi_{\Gamma}(\boldsymbol{\xi}_\epsilon(\mathbf{z}'))=\phi_\Gamma(\boldsymbol{\xi}_\epsilon(\mathbf{z}')),~~~\mathbf{z}\in T_\epsilon.
\end{equation}
In particular, suppose $\mathbf{m}$ is the solution 
to the Optimal transport problem in Equation~\eqref{eq:OT} on $\Gamma$. 
Then, we will show that
\begin{equation}
 \tilde{\mathbf{m}}(\mathbf{z}) := (P_\Gamma\mathbf{z}+\phi_\Gamma(\mathbf{z})\mathbf{n}(\mathbf{z}))~\mathbf{m}(P_\Gamma\mathbf{z})  
\end{equation}
is a solution to the extended Optimal Transport problem in Equation~\eqref{eq:OT2}.

The first step is to show that if $\boldsymbol{\xi}\#\mu = \nu$ with $\mu=fdS$ and $\nu=gdS$, then
 $\boldsymbol{\xi}_\epsilon \mu =\nu$, with $\mu = f_\epsilon d\mathbf{x}$, and $\nu = g_\epsilon d\mathbf{x}$. 
 But this is another exercise on the coarea formula. Let $E\subset T_\epsilon$, we have

\begin{align*}
    \int_E \overline{g}(\mathbf{x}) d\mathbf{x} 
    &=\int_{-\epsilon}^\epsilon \int_{E\cap \Gamma_\eta } g(P_\Gamma\mathbf{z}') J(\mathbf{z}')~dSd\eta \\
    &= \int_{-\epsilon}^\epsilon \int_{P_\Gamma(E\cap\Gamma_\eta)} g(\mathbf{z})~dS d\eta\\
    &= \int_{-\epsilon}^\epsilon \int_{\boldsymbol{\xi}^{-1}(P_\Gamma(E\cap\Gamma_\eta))} f(\mathbf{z})~dS d\eta\\
    &= \int_{-\epsilon}^{\epsilon} \int_{P^{-1}_{\Gamma, \Gamma_\eta} (\boldsymbol{\xi}^{-1}(P_{\Gamma}(E \cap \Gamma_\eta)))} f(P_{\Gamma} \mathbf{z})J(\mathbf{z})~dS d\eta \\
    &=  \int_{-\epsilon}^\epsilon \int_{\boldsymbol{\xi}_\epsilon^{-1}(E\cap \Gamma_\eta) } f(P_\Gamma \mathbf{z}')J(\mathbf{z}')~dSd\eta
    = \int_{\boldsymbol{\xi}_\epsilon^{-1}(E)} \overline{f}(\mathbf{x}) d\mathbf{x},
\end{align*}
where the penultimate equality follows since $P_{\Gamma, \Gamma_\eta}^{-1} \circ \boldsymbol{\xi}^{-1} \circ P_{\Gamma} \circ P^{-1}_{\Gamma, \Gamma_\eta} \circ \boldsymbol{\xi} \circ P_{\Gamma} = P_{\Gamma, \Gamma_\eta}^{-1} \circ \boldsymbol{\xi}^{-1} \circ \boldsymbol{\xi} \circ P_{\Gamma} = P_{\Gamma, \Gamma_\eta}^{-1} \circ P_{\Gamma} = \text{Id}$, we then know that $P^{-1}_{\Gamma, \Gamma_\eta} (\boldsymbol{\xi}^{-1}(P_{\Gamma}(E \cap \Gamma_\eta))) = \boldsymbol{\xi}_{\epsilon}^{-1}(E \cap \Gamma_\eta)$.


Comparing Equation~\eqref{eq:OT2} with Equation~\eqref{extended-cost}, Equation~\eqref{eq:int_Gamma=int_Gamma_y}, and Equation~\eqref{eq:densities}, we see that the transport cost over all $\boldsymbol{\xi}_\epsilon$ is minimized for $\boldsymbol{\xi}_\epsilon = \tilde{\mathbf{m}}$:
\begin{equation}\label{eq:equivalent_integrals}
\begin{aligned}
        &\int_{T_{\epsilon}} c_{\sigma}\left(\mathbf{z}',\boldsymbol{\xi}_\epsilon(\mathbf{z}') \right) \overline{f}(\mathbf{z}') d\mathbf{z}' \\
    &=\int_{-\epsilon}^\epsilon \int_{\Gamma_\eta} c(P_\Gamma  \mathbf{z}', P_\Gamma \boldsymbol{\xi}_\epsilon(\mathbf{z}') ) f(P_{\Gamma}\mathbf{z}')J(\mathbf{z}')dSd\eta \\
    &= 2\epsilon \int_\Gamma c(\mathbf{z}, \boldsymbol{\xi}(\mathbf{z})) f(\mathbf{z}) dS\\
    &\ge 2\epsilon \int_\Gamma c(\mathbf{z}, \boldsymbol{m}(\mathbf{z})) f(\mathbf{z}) dS = \int_{T_{\epsilon}} c_{\sigma}\left(\mathbf{z}',\tilde{\boldsymbol{m}}_\epsilon(\mathbf{z}') \right) \overline{f}(\mathbf{z}') d\mathbf{z}',
\end{aligned}
\end{equation}


Following the above construction, we see that Equation~\eqref{eq:Pm_eps=m(P)} holds for $\tilde{\mathbf{m}}$.
Now consider $\mathbf{z}_y\in \Gamma_{y}$.  
By construction, $\tilde{\mathbf{m}}(\mathbf{z}_y)\in\Gamma_\eta,$ implying Equation~\eqref{eq:phi_z=phim_z}.  
$\tilde{\mathbf{m}}(\mathbf{z})$  does not move mass in the normal direction of $\Gamma$. 

For our extended problem, since the source and target densities are supported on $\mathbb{R}^3$, the Optimal Mapping $\overline{\mathbf{m}}$ and potential function $v$ satisfy:
\begin{equation}\label{eq:mappingEuclidean}
\nabla v(\mathbf{z}) = -\nabla_{\mathbf{z}} c_{\sigma}(\mathbf{z},\boldsymbol{\xi}), 
\end{equation}
for $\boldsymbol{\xi}=\overline{\mathbf{m}}(\mathbf{z})$, where we emphasize that $\nabla$ is the Euclidean gradient and $v$ is $c_{\sigma}$-convex. 
The relation defined in Equation~\eqref{eq:mappingEuclidean} applied to $\tilde{\mathbf{m}}$ and separated into one defined on surfaces that are equidistant to $\Gamma$ becomes
\begin{equation}\label{eq:pde-on-Gamma_eta}
    (I-\mathbf{n}\otimes\mathbf{n})\nabla \tilde{v} (\mathbf{z}) = -(I-\mathbf{n}\otimes\mathbf{n}) \nabla_{\mathbf{z}} c \left(P_\Gamma \mathbf{z}, \mathbf{m} \right),
\end{equation}
and one in the normal direction of the surface,
\begin{equation}\label{eq:normal-derivative-u_epsilon}
    \nabla \tilde{v} (\mathbf{z}) \cdot \mathbf{n} = \sigma \left(\phi_\Gamma(\tilde{\mathbf{m}}(\mathbf{z})) - \phi_\Gamma(\mathbf{z}) \right). 
\end{equation}
Here $\mathbf{n}$ is the normal of $\Gamma$ at $P_\Gamma\mathbf{z}$.
Notice that for $\mathbf{z}\in\Gamma$, Equation~\eqref{eq:pde-on-Gamma_eta} is Equation~\eqref{eq:relation}. This means that the restriction of $v$ on $\Gamma$ solves Equation~\eqref{eq:relation}.

Using
Equation~\eqref{eq:phi_z=phim_z} in Equation~\eqref{eq:normal-derivative-u_epsilon}, we obtain
\begin{equation}\label{eq:indepsystem2}
    \frac{\partial \tilde{v}}{\partial \mathbf{n}}(\mathbf{z}) = 0,~~~\mathbf{z}\in T_\epsilon.
\end{equation}
Then, by Equation~\eqref{eq:pde-on-Gamma_eta} 
$$\tilde{v}(\mathbf{z}) = u(P_{\Gamma}\mathbf{z})+C,$$ for some constant $C$ and $P_{\Gamma} \overline{\mathbf{m}}(\mathbf{z}) = \mathbf{m}(P_{\Gamma}\mathbf{z})$, which is the solution pair from the Optimal Transport problem on $\Gamma$.

Now, in order to finish the proof of Theorem~\ref{thm:maintheorem}, it remains to show that $\tilde{v}$ is $\bar{c}$-convex. Once this is shown, then by Theorem~\ref{thm:uniqueness} there is a unique solution of the Optimal Transport problem in Equation~\eqref{eq:OT2}, then we have shown, in fact, $\tilde{\mathbf{m}}=\overline{\mathbf{m}}$ satisfying $\phi_{\Gamma}(\overline{\mathbf{m}}(\mathbf{z})) = \phi_{\Gamma}(\mathbf{z})$ and $P_{\Gamma} \overline{\mathbf{m}}(\mathbf{z}) = \mathbf{m}(P_{\Gamma}\mathbf{z})$.
 
\subsubsection*{$c_{\sigma}$-Convexity of the Potential Function}

\begin{lemma}\label{lem:cconvex}
$\tilde{v}(\mathbf{z}) = u(P_{\Gamma}\mathbf{z})$ is $\bar{c}$-convex.
\begin{proof}
Recall the definition of $c$-convexity from Definition~\ref{def:cconvex}: $u$ is $c$-convex since it is the potential function for the Optimal Transport problem on $\Gamma$. This means that for all $\mathbf{x} \in \Gamma$, there exists a point $\mathbf{\tilde{x}} \in \Gamma$ and the $c$-transform of $u$ satisfies $u^{c}(\mathbf{\tilde{x}})$ such that:
\begin{align}\label{eq:xdir}
-u^{c}(\mathbf{\tilde{x}}) - c(\mathbf{x},\mathbf{\tilde{x}}) &= u(\mathbf{x}), \\
-u^{c}(\mathbf{\tilde{x}}) - c(\mathbf{x}', \mathbf{\tilde{x}}) &\leq u(\mathbf{x}'), \ \forall \mathbf{x}' \in \Gamma.\label{eq:xdir2}
\end{align}

Now, fix $\mathbf{x} \in \Gamma$ and fix a $\mathbf{z} \in T_{\epsilon}$ such that $P_{\Gamma}\mathbf{z} = \mathbf{x}$. From Equation~\eqref{eq:xdir}, given $\mathbf{x} \in \Gamma$, we have an $\tilde{\mathbf{x}} \in \Gamma$ such that Equation~\eqref{eq:xdir} holds. Choose $\tilde{\mathbf{z}}$ such that $P_{\Gamma}\tilde{\mathbf{x}} = \tilde{\mathbf{x}}$ and $ \phi_{\Gamma}\tilde{\mathbf{z}} = \phi_{\Gamma} \mathbf{z} $. By the definition of the $c_{\sigma}$-transform in $T_{\epsilon}$:
\begin{equation}
\tilde{v}^{c_{\sigma}}(\mathbf{\tilde{z}}) = \sup_{\boldsymbol{\xi}} \left(-c_{\sigma}(\boldsymbol{\xi}, \mathbf{\tilde{z}} ) - \tilde{v}(\boldsymbol{\xi}) - C \right),
\end{equation}

\begin{equation}
= \sup_{P_{\Gamma}\boldsymbol{
\xi, \phi_{\Gamma}\boldsymbol{\xi}
}} \left( -c(P_{\Gamma} \boldsymbol{\xi}, \mathbf{\tilde{x}}) - \frac{\sigma}{2} (\phi_{\Gamma} \boldsymbol{\xi}-\phi_{\Gamma} \tilde{\mathbf{z}})^2 - u(P_{\Gamma} \boldsymbol{\xi}) - C \right),
\end{equation}
we can immediately take the supremum over $\phi_{\Gamma} \boldsymbol{\xi}$ by choosing $\phi_{\Gamma} \boldsymbol{\xi} = \phi_{\Gamma} \tilde{\mathbf{z}}$, since the other terms do not depend on $\phi_{\Gamma} \boldsymbol{\xi}$. Thus,
\begin{equation}
\tilde{v}^{c_{\sigma}}(\mathbf{\tilde{z}}) = \sup_{P_{\Gamma}\boldsymbol{\xi}} \left(-c(P_{\Gamma}\boldsymbol{\xi}, \mathbf{\tilde{x}} ) - u(P_{\Gamma}\boldsymbol{\xi})-C \right) = u^{c}(\mathbf{\tilde{x}}) - C.
\end{equation}

Thus, we get:
\begin{equation}
-\tilde{v}^{c_{\sigma}}(\mathbf{\tilde{z}}) - c_{\sigma}(\mathbf{z},\mathbf{\tilde{z}}) = -u^{c}(\mathbf{\tilde{x}}) +C- c(P_{\Gamma} \mathbf{z}, \mathbf{\tilde{x}}) = u(P_{\Gamma}\mathbf{z})+C = \tilde{v}(\mathbf{z}),
\end{equation}
by Equation~\eqref{eq:xdir}. Now, for any $\mathbf{z}' \in T_{\epsilon}$. Then,
\begin{multline}
-\tilde{v}^{c_{\sigma}}(\mathbf{\tilde{z}}) - c_{\sigma}(\mathbf{z}', \mathbf{\tilde{z}}) = -u^{c}(\mathbf{\tilde{x}})+C -\frac{\sigma}{2}(\phi_{\Gamma} \mathbf{z}'- \phi_{\Gamma} \tilde{\mathbf{z}})^2- c(P_{\Gamma}\mathbf{z}', P_{\Gamma} \tilde{\mathbf{z}}) \leq \\
-u^{c}(\mathbf{\tilde{x}})+C - c(P_{\Gamma} \mathbf{z}', P_{\Gamma} \tilde{\mathbf{z}}) \leq u(P_{\Gamma} \mathbf{z}')+C = \tilde{v}(\mathbf{z}'), \ \forall \mathbf{z}',
\end{multline}
by Equation~\eqref{eq:xdir2}. Thus, $\tilde{v}$ is $c_{\sigma}$-convex.

\end{proof}
\end{lemma}

\subsection{Example: Optimal Transport on the Unit Sphere}\label{sec:sphere}

In our numerical implementation in Section~\ref{sec:examples}, we have found it fruitful to utilize explicit expressions for the Optimal Transport mapping in terms of the Euclidean gradient of the potential function. We have also found it useful on the sphere to simply use explicit expressions for the mixed Hessian terms in the PDE~\eqref{eq:OTExtension}. Expressions for the mixed Hessian are much more difficult to derive on other surfaces, so in those cases the simple discretization presented in Section~\ref{sec:examples} is used instead. In practice, it is perfectly reasonable to eschew these analytical derivations and simply discretize the PDE system~\eqref{eq:OTExtension} directly.

Let $\Gamma$ be the unit sphere in $\mathbb{R}^3$. 
Although we will be using the usual spherical polar coordinates, $\mathbf{z}= (r \cos \phi \sin \theta, r \sin \phi \sin \theta, r \cos \theta),$ for $0\le\phi<2\pi$, $0\le\theta\le \pi$, on the sphere we will not need to derive many quantities that depend on the coordinates $\theta$ and $\phi$. One minor exception is
the surface area element $dS=r\sin(\theta)d\phi d\theta$ which will be useful for integrating the source and target densities $f$ and $g$ when they are explicitly known. Some of the important quantities that do not depend on $\theta$ and $\phi$ are:
$$
\phi_\Gamma( \mathbf{z})= r-1,~P_\Gamma \mathbf{z} = \frac{ \mathbf{z}}{|| \mathbf{z}||},~J( \mathbf{z})= 1-\frac{\phi_\Gamma( \mathbf{z})}{1+ \phi_\Gamma(\mathbf{z})}=\frac{1}{r^2}.
$$
Thus densities in $\Theta_{T_\epsilon}$ take the form
$$\overline{\rho}(r, \phi,\theta)= \frac{\rho(\phi,\theta)}{2\epsilon r^2}.$$

\subsubsection*{Case 1: $c(\mathbf{x},\mathbf{y}) = \frac{1}{2}d_{\mathbb{S}^2}(\mathbf{x},\mathbf{y})$ is the the squared geodesic distance between $\mathbf{x}$ and $\mathbf{y}$ on the unit sphere.} 
The extended cost function on the unit sphere has the explicit form $c(\mathbf{x}, \mathbf{y}) = \frac{1}{2}\arccos(\mathbf{x} \cdot \mathbf{y})^2$, and consequently we have:
\begin{equation}
    c_{\sigma}(\mathbf{z}, \boldsymbol{\xi}) = \frac{\sigma}{2}\left(||\mathbf{z}||-||\boldsymbol{\xi}||\right)^2 + \frac{1}{2}\arccos \left( \frac{\mathbf{z}}{||\mathbf{z}||} \cdot \frac{\boldsymbol{\xi}}{||\boldsymbol{\xi}||} \right)^2,
\end{equation}

Once we have chosen a cost function, we can compute the Optimal Transport mapping. The Optimal Transport mapping $\overline{\mathbf{m}}(\mathbf{z})$ can be found by solving Equations~\eqref{eq:pde-on-Gamma_eta}and~\eqref{eq:normal-derivative-u_epsilon}. Thus, we find:
\begin{equation}\label{eq:msq1}
    \phi_{\Gamma}(\overline{\mathbf{m}}(\mathbf{z})) = \phi_{\Gamma}(\mathbf{z}) + \frac{1}{\sigma}\nabla v(\mathbf{z}) \cdot \mathbf{n},
\end{equation}
and
\begin{equation}\label{eq:msq2}
    P_{\Gamma}(\overline{\mathbf{m}}(\mathbf{z})) = \text{exp}_{P_{\Gamma}(\mathbf{z})} \left( (1+\phi_{\Gamma}(\mathbf{z})) \left( I - \mathbf{n} \otimes \mathbf{n} \right) \nabla v(\mathbf{z}) \right),
\end{equation}
which states that the projection of $\nabla v$ onto the tangent plane is multiplied by the factor $1+\phi_{\Gamma}(\mathbf{z})$ and then transported via the exponential map on $\mathbb{S}^2$ starting at the point $P_{\Gamma}(\mathbf{z})$ in the direction of the projection of $\nabla v$ onto the tangent plane at $P_{\Gamma}(\mathbf{z})$. This leads to the expression:
\begin{equation}\label{eq:msq3}
    \overline{\mathbf{m}}(\mathbf{z}) = P_{\Gamma}(\overline{\mathbf{m}}(\mathbf{z})) + \phi_{\Gamma} \left(\overline{\mathbf{m}}(\mathbf{z}) \right)\mathbf{n}.
\end{equation}


The mixed Hessian term in $T_{\epsilon}$ can be derived by using the mixed Hessian term for the squared geodesic cost on the unit sphere~\cite{HT_OTonSphere2},~\cite{HT_OTonSphere3}:
\begin{equation}\label{eq:mHsqgd}
    \left\vert \det D^2_{\mathbf{x}\mathbf{y}, \mathbb{S}^2} \left( \frac{1}{2}d_{\mathbb{S}^2}(\mathbf{x}, \mathbf{y})^2 \right) \right\vert = \frac{\left\Vert \nabla_{\mathbb{S}^2} u(\mathbf{x}) \right\Vert}{\sin \left(\left\Vert \nabla_{\mathbb{S}^2} u(\mathbf{x}) \right\Vert \right)}.
\end{equation}
Denoting $\nabla_{\Gamma_{\eta}} v(\mathbf{z}) = (I-\mathbf{n} \otimes \mathbf{n})\nabla v(\mathbf{z})$, the corresponding mixed Hessian for the extended problem is:
\begin{equation}\label{eq:mHfullsqgd}
    \left\vert \det D^2_{\mathbf{z} \boldsymbol{\xi}} c_{\sigma}(\mathbf{z}, \boldsymbol{\xi}) \right\vert = \sigma \frac{\left\Vert (1+\phi_{\Gamma}(\mathbf{z})) \nabla_{\Gamma_{\eta}} v(\mathbf{z}) \right\Vert}{(1+\phi_{\Gamma}(\mathbf{z}))^2 \sin \left(\left\Vert (1+\phi_{\Gamma}(\mathbf{z})) \nabla_{\Gamma_{\eta}} v(\mathbf{z}) \right\Vert \right)},
\end{equation}
for $\boldsymbol{\xi} = \overline{\mathbf{m}}(\mathbf{z})=\mathbf{m}(P_\Gamma\mathbf{z})$.

\subsubsection*{Case 2:  $c(\mathbf{x},\mathbf{y}) = -\log(1 - \mathbf{x} \cdot \mathbf{y})$, the logarithmic cost function appearing in the reflector antenna problem.} 
We have
\begin{equation}
    c_{\sigma}(\mathbf{z}, \boldsymbol{\xi}) = \frac{\sigma}{2}\left(||\mathbf{z}||-||\boldsymbol{\xi}||\right)^2 - \log \left( 1 - \frac{\mathbf{z}}{||\mathbf{z}||} \cdot \frac{\boldsymbol{\xi}}{||\boldsymbol{\xi}||} \right)^2.
\end{equation}

The Optimal Transport mapping $\overline{\mathbf{m}}(\mathbf{z})$ can be found by solving Equations~\eqref{eq:pde-on-Gamma_eta}and~\eqref{eq:normal-derivative-u_epsilon}. As in the squared geodesic cost case, we find:
\begin{equation}\label{eq:lc1}
    \phi_{\Gamma}(\overline{\mathbf{m}}(\mathbf{z})) = \phi_{\Gamma}(\mathbf{z}) + \frac{1}{\sigma}\nabla v(\mathbf{z}) \cdot \mathbf{n},
\end{equation}
and the optimal transport mapping satisfies ~\cite{HT_OTonSphere3}:
\begin{multline}\label{eq:lc2}
    P_{\Gamma}\left(\overline{\mathbf{m}}(\mathbf{z})\right) = P_{\Gamma}\mathbf{z} \frac{\left\Vert (1+\phi_{\Gamma}(\mathbf{z})) \nabla_{\Gamma_{\eta}} v(\mathbf{z}) \right\Vert^2 - 1}{\left\Vert (1+\phi_{\Gamma}(\mathbf{z})) \nabla_{\Gamma_{\eta}} v(\mathbf{z}) \right\Vert^2 + 1} \\ - (1+\phi_{\Gamma}(\mathbf{z})) \nabla_{\Gamma_{\eta}} v(\mathbf{z}) \frac{2}{\left\Vert (1+\phi_{\Gamma}(\mathbf{z})) \nabla_{\Gamma_{\eta}} v(\mathbf{z}) \right\Vert^2 + 1}.
\end{multline}

As with the squared geodesic cost function, this leads to the expression:
\begin{equation}\label{eq:lc3}
    \overline{\mathbf{m}}(\mathbf{z}) = P_{\Gamma}(\overline{\mathbf{m}}(\mathbf{z})) + \phi_{\Gamma} \left(\overline{\mathbf{m}}(\mathbf{z}) \right)\mathbf{n}.
\end{equation}

In this case, the mixed Hessian term for the logarithmic cost on the sphere is:
\begin{equation}\label{eq:mHlogcost}
    \left\vert \det D^2_{\mathbf{x}\mathbf{y}, \mathbb{S}^2} \left( -\log(1 - \mathbf{x} \cdot \mathbf{y}) \right) \right\vert = \frac{\left(\left\Vert \nabla_{\mathbb{S}^2} u(\mathbf{x}) \right\Vert^2+1 \right)^2}{4},
\end{equation}
for $\boldsymbol{\xi} = \overline{\mathbf{m}}(\mathbf{z})$. The corresponding mixed Hessian for the extended problem on $T_{\epsilon}$ is:
\begin{equation}\label{eq:mHfulllogcost}
    \left\vert \det D^2_{\mathbf{z} \boldsymbol{\xi}} c_{\sigma}(\mathbf{z}, \boldsymbol{\xi}) \right\vert = \frac{\sigma \left( \left\Vert (1+\phi_{\Gamma}(\mathbf{z})) \nabla_{\Gamma_{\eta}} v(\mathbf{z}) \right\Vert^2+1 \right)^2}{4(1+\phi_{\Gamma}(\mathbf{z}))^2},
\end{equation}
for $\boldsymbol{\xi} = \overline{\mathbf{m}}(\mathbf{z})$.


\subsection{Optimal Transport on Surfaces with Boundary}\label{sec:boundary}

Here we outline how the proposed volumetric approach may be generalized to surfaces with smooth boundaries. 
Let $\Gamma$ be the boundary of a subset $\Omega\subset\mathbb{R}^3$. Let $\Gamma_1$ and $\Gamma_2$ be two open subsets of $\Gamma$ such that 
$\partial\Gamma_1 =\partial\Gamma_2$ is a closed smooth curve and 
\[
\Gamma= \overline{\Gamma_1\cup\Gamma_2}.
\] 
Take $\Gamma$ for a sphere for instance, $\Gamma_1$ and $\Gamma_2$ can be respectively the northern and the southern hemisphere, and $\partial\Gamma_1=\partial\Gamma_2$ is the equator.

Consider the following partitioning of $T_\epsilon$:
\begin{equation*}
    T_\epsilon:= \{ \mathbf{x}\in\mathbb{R}^3: \min_{\mathbf{z}\in\Gamma_1} |\mathbf{x}-\mathbf{z}|<\epsilon\}= T_\epsilon^1\cup Q_\epsilon,
\end{equation*}
where
\begin{equation*}
    Q_\epsilon = \{ \mathbf{z}\in T_\epsilon: P_\Gamma\mathbf{z} \in\partial\Gamma_1  \}.
\end{equation*}

Starting from the optimization problem involving the source and target densities 
$f, g: \Gamma_1\mapsto \mathbb{R}$ satisfying
\[
\text{supp}(f)\subseteq\text{supp}(g)=\Gamma_1,
\]
we will formally arrive at the same optimal transport PDEs defined in \eqref{eq:OTExtension}, but posed in $T_\epsilon^1$ instead of $T_\epsilon$, and with zero Neumann boundary imposed on $\partial T_\epsilon^1$.
Notice that the essential ingredients used in the extension, $P_{\Gamma_1}\equiv P_\Gamma$ and $J_{\Gamma_1}\equiv J_\Gamma$ in $T_\epsilon^1$, so the extended optimal transport problem can be defined without further complication. 

Following the theory developed in \cite{Urbas}\cite{HamfeldtBVP2}, we will then impose the zero Neumann condition on $\partial T_\epsilon^1$. Of course on $\partial T^1_\epsilon\cap\partial T_\epsilon$, we can apply a similar procedure described in \eqref{interp closure}. On the face $\partial  T^1_\epsilon\cap \partial Q_\epsilon$, we can adopt the standard finite difference approach to enforce Neumann condition; e.g. \cite{bedrossian2010second, hellrung2012second}.

\begin{figure}
    \includegraphics[width=0.8\textwidth]{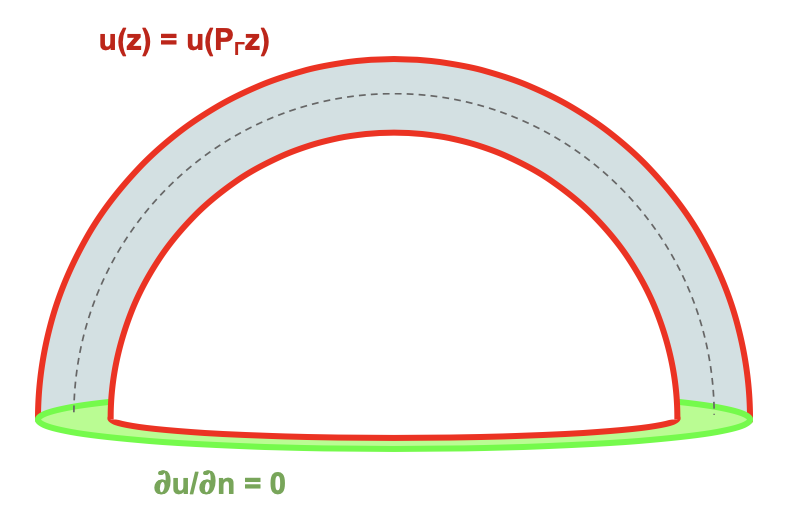}
	\caption{Different boundary conditions are applied for different parts of the boundary. For $\partial T_{\epsilon}^{1} \cap \partial T_{\epsilon}$, we apply the boundary condition $u(\mathbf{z}) = u(P_{\Gamma}\mathbf{z})$ (depicted in red). For the rest of the boundary, we extend the natural zero-Neumann condition on the boundary of $\Gamma$ to zero-Neumann conditions on the corresponding part of $\partial T_{\epsilon}^{1} \cap \partial Q_{\epsilon}$ (depicted in green).}\label{fig:boundary}
\end{figure}


\section{Computational Examples}\label{sec:examples}
The key innovation in reformulating the Optimal Transport problem of $\Gamma$ to one on $T_{\epsilon}$ is to allow for a wide selection of numerical discretizations for the PDE~\eqref{eq:OTExtension} defined on $T_{\epsilon}$. In this section, we present a simple method using uniform Cartesian grids and demonstrate some computational results using the method.

We will discretize \eqref{eq:OTExtension} and compute most of the needed geometric information about $\Gamma$ on the part of a uniform Cartesian grid lying within

\begin{equation*}
\mathcal{T}_{\epsilon}^h = h\mathbb{Z}^3\cap T_\epsilon,
\end{equation*}
which will consist of interior points, where we perform a full PDE solve using the PDE~\eqref{eq:OTExtension} and boundary points, where we exploit the fact that the potential function is \textit{a priori} known to be constant in the normal direction. 

The grid is assumed to resolve the geometry of $\Gamma$, meaning that the spacing $h$ is smaller than the reach of $\Gamma$.  
We will solve the discretized system by the following iterative scheme: for all $\mathbf{z}\in \mathcal{T}_\epsilon^h$,
\begin{equation}\label{eq:momentum-method}
\begin{aligned}
    v^n_E(\mathbf{z}) &:=v^n(\mathbf{z})+\gamma_n\left(v^n(\mathbf{z})-v^{n-1}(\mathbf{z})\right),\\
    v^{n+1}(\mathbf{z}) &:= v^n_{E} -\Delta t F_h(v^n_{E}, \mathbf{z}, \mathcal{N}_h(\mathbf{z}));
\end{aligned}    
\end{equation}

Here $F_h$ is a finite difference discretization of \eqref{eq:OTExtension}-\eqref{eq:bc} and $\mathcal{N}_h(\mathbf{z})$ is the set of grid nodes around $\mathbf{z}$ forming the finite difference stencil for the discretization, and $\gamma_n$ are small positive numbers, chosen to accelerate the scheme.
In Section~\ref{sec:geometry-proc} and Section~\ref{sec:descriptionofscheme}, we will describe the ingredients in $G_h$ used in the numerical experiments in detail.

We will be demonstrating some computational results for the unit sphere, the northern hemisphere of the unit sphere, and for a $2$-torus in $\mathbb{R}^3$.
On the unit sphere, we will employ the squared geodesic cost function $c(\mathbf{x},\mathbf{y}) = \frac{1}{2}d_{\mathbb{S}^2}(\mathbf{x},\mathbf{y})^2$ as well as the logarithmic cost function $c(\mathbf{x}, \mathbf{y}) = -\log(1-\mathbf{x} \cdot \mathbf{y})$ arising in the reflector antenna problem, see~\cite{Wang_Reflector, Wang_Reflector2}. To make the reflector antenna computation more physically realistic, we also perform a computation on the northern hemisphere of the unit sphere using the logarithmic cost function. Finally, to demonstrate the versatility of the extension method over other surfaces, we perform a computation on a torus $\Gamma \subset \mathbb{R}^3$ using the Euclidean cost $c(x,y) = \frac{1}{2}\left\Vert x - y \right\Vert^2$ on the base manifold $\Gamma$. We stress here that we do not use the standard Euclidean cost in $\mathbb{R}^3$, but construct the full extended cost $c_{\sigma}$ using $c(x,y) = \frac{1}{2}\left\Vert x - y \right\Vert^2$ defined for $x, y \in \Gamma$ and add the $\sigma$-penalty term as explained in Section~\ref{sec:costs}.


\subsection{Geometry processing}\label{sec:geometry-proc}
For surfaces defined by level set functions, the computation of the closest point mapping $P_{\Gamma}$ can be done via standard level set methods, see~\cite{ahmed, chengtsai, sethian, zhangzhaoqian}. For surfaces defined by parametrized patches, methods such as those developed in~\cite{tsaidistance} can be used in combination with a KD-tree. As shown in~\cite{tsai}, 
the quantity $J(z) = \sigma_1(z) \sigma_2(z)$, where $\sigma_1$ and $\sigma_2$ are the two largest singular values of $D P_{\Gamma}$, which can be computed using finite differences.

We recognize that, in contrast to working directly with parameterized surfaces, 
the reach of $\Gamma$ imposes an important limitation to the mesh size used to discretize $T_\epsilon,$ independent of the smoothness of the quantities that we wish to compute in $T_\epsilon.$ One may resort to adaptive meshing techniques, e.g. \cite{min2007geometric, min2007second}, for surfaces with small reach.

\subsection{Discretization of the PDEs}\label{sec:descriptionofscheme}

For any function $f: T_{\epsilon} \rightarrow \mathbb{R}$, we use the following standard centered-difference discretization $\mathcal{D}^h_{1}f(\mathbf{x}_i)$, $\mathcal{D}^h_{2}f(\mathbf{x}_i)$, and $\mathcal{D}^h_{3}f(\mathbf{x}_i)$ for the first-order derivatives $f_x, f_y, f_z$ at a point $\mathbf{x}_i$, respectively. Letting $\mathbf{e}_{1}, \mathbf{e}_{2}, \mathbf{e}_{3}$ denote $(1,0,0), (0,1,0)$, and $(0,0,1)$, respectively, we compute our first-order derivatives as follows:
\begin{equation*}
\mathcal{D}^h_{j} f(\mathbf{x}_i) = \frac{f(\mathbf{x}_i+h\mathbf{e}_{j})-f(\mathbf{x}_i-h\mathbf{e}_{j})}{2h}.
\end{equation*}
Correspondingly, $\nabla f$ is approximated by 
\begin{equation}\label{eq:gradient2}
\nabla^h f(\mathbf{x}_i) = \mathcal{D}^h_{1} f(\mathbf{x}_i) \mathbf{e}_1 + \mathcal{D}^h_{2} f(\mathbf{x}_i) \mathbf{e}_2 +\mathcal{D}^h_{3} f(\mathbf{x}_i) \mathbf{e}_3.
\end{equation}
The second-order derivatives are computed by:
\begin{equation*}
\mathcal{D}^{h}_{jj} f(\mathbf{x}_i) = \frac{f(\mathbf{x}_i+h\mathbf{e}_{j})-2f(\mathbf{x}_i)+f(\mathbf{x}_i-h\mathbf{e}_{j})}{h^2},
\end{equation*}
and the following discretization for mixed second-order derivatives:
\begin{multline*}
\mathcal{D}^{h}_{jk} f(\mathbf{x}_i) = \frac{f(\mathbf{x}_i+h(\mathbf{e}_{j}+\mathbf{e}_{k}))+ f(\mathbf{x}_i-h(\mathbf{e}_{j}+\mathbf{e}_{k}))}{h^2} - \\
\frac{f(\mathbf{x}_i+h(-\mathbf{e}_{j}+\mathbf{e}_{k}))+f(\mathbf{x}_i+h(\mathbf{e}_{j}-\mathbf{e}_{k}))}{h^2}.
\end{multline*}
These second-order derivatives are used in the computation of terms  in the determinant of a Hessian matrix:
\begin{equation*}
    \det D^2 F = F_{xx}\left( F_{yy}F_{zz}-F_{yz}^2 \right) - F_{xy}\left( F_{xy}F_{zz}-F_{xz}F_{yz} \right) + F_{xz} \left(F_{xy}F_{yz} - F_{xz}F_{yy} \right).
\end{equation*}
We will use $\det D^2_h F(\mathbf{z})$ to denote the resulting finite difference approximation of $\det D^2 F(\mathbf{z})$.

These finite differences collectively define 
the computational neighborhood $\mathcal{N}^h(\mathbf{x}_i)$ for each grid node 
$\mathbf{x}_i\in \mathcal{T}_\epsilon^h$.
Hence $\mathcal{N}^h(\mathbf{x}_i)$ contains 
 $18$ grid nodes which are no more than a distance $\sqrt{2}h$ from $\mathbf{x}_i$.

The system of PDEs~\eqref{eq:OTExtension} represents four equations in four unknowns: $v$, and the coordinates of $\overline{\mathbf{m}}$. A first possible method would be direct discretization of these four equations. However, since we will be able to solve for the Optimal Transport mapping in terms of the potential function on the base manifold, we opt for the expedient of simply devising a numerical discretization of the first equation in the system of PDEs~\eqref{eq:OTExtension} and deriving an approximation of the mapping which is inspired by Theorem~\ref{thm:maintheorem}. More specifically, we assume that for the given surface $\Gamma$ and cost function $c$, we can solve for the mapping on the base manifold in terms of the potential function on the base manifold in Equation~\eqref{eq:OTGamma2}. Then, for our discretization of the PDE system~\eqref{eq:OTExtension}, we will simply employ Equation~\eqref{eq:Pm_eps=m(P)} to construct the tangential coordinate of the mapping, by using interpolated values of the discrete gradient of the potential function on the base manifold, and then account for how far the mapping goes in the normal direction via the normal derivative using Equation~\eqref{eq:normal-derivative-u_epsilon}.

From Equation~\eqref{eq:OTGamma2}, let $\boldsymbol{\alpha}$ denote the map that satisfies $\mathbf{m} = \boldsymbol{\alpha}(\mathbf{p})$, where $\mathbf{p} = \nabla_{\Gamma}u(\mathbf{x})$. That is, suppose for any desired cost function on $\Gamma$, we have analytically derived an expression for the Optimal Transport mapping $\mathbf{m}$ on $\Gamma$ as a function of the gradient of the potential function $\nabla u(\mathbf{x})$ on $\Gamma$. 

We introduce a trilinear interpolation operator $\mathcal{I}_h$:  $\mathcal{I}_h[v(\mathbf{z})]$ denotes the interpolated values of $v$ at the point $\mathbf{z} \in \Gamma$ using the grid nodes surrounding $\mathbf{z}$.
Of course, if $\mathbf{z}$ is on a grid node, no interpolation is needed.
We will be first using the trilinear operator to approximate the Optimal Transport mapping by performing a trilinear interpolation of the projection of the gradient on a local tangent plane computed at the point $P_{\Gamma}\mathbf{x}_i$ and then solving Equation~\eqref{eq:pde-on-Gamma_eta} for the mapping, for which we assume that we have analytical expressions.

The projection of the finite difference gradient in the normal direction is obtained via the equation
\begin{equation}\label{eq:gradienttangent}
    \mathcal{D}_{\mathbf{n}}^{h} v(\mathbf{x}_i):= \left(\nabla^h v\cdot\mathbf{n}(\mathbf{x}_i)\right)\mathbf{n}(\mathbf{x}_i) \equiv \mathbf{n}\otimes\mathbf{n}(\mathbf{x}_i)  \nabla^h v(\mathbf{x}_i)
\end{equation}
and the projection of the gradient onto the local tangent plane via

\begin{equation}
\mathcal{D}_{\mathbf{n}^{\perp}}^h v(\mathbf{x}_i) = \left(I-\mathbf{n}\otimes\mathbf{n}(\mathbf{x})\right)\nabla^h v(\mathbf{x})
\end{equation}

Then, we will compute the following interpolation of the gradient:
\begin{equation}
    \overline{\mathcal{D}}^{h}_{\mathbf{n}^{\perp}} v(\mathbf{x}_i) := \mathcal{I}_{h} \left[ \mathcal{D}^{h}_{\mathbf{n}^{\perp}} v(P_{\Gamma} \mathbf{x}_i) \right].
\end{equation}

We then define the approximation of the Optimal Transport mapping on $\mathcal{T}_\epsilon^h$ as follows:
\begin{equation}
    \mathbf{m}^h(\mathbf{x}_i) := \boldsymbol{\alpha} \left( \overline{\mathcal{D}}^{h}_{\mathbf{n}^{\perp}} v(\mathbf{x}_i) \right) + \mathbf{n} \left( \boldsymbol{\alpha} \left( \overline{\mathcal{D}}^{h}_{\mathbf{n}^{\perp}} v(\mathbf{x}_i) \right) \right) \left( \phi(\mathbf{x}_i) + \frac{\nabla^{h}v(\mathbf{x}_i) \cdot \mathbf{n}(\mathbf{x}_i)}{\sigma} \right),
\end{equation}
which is constructed so that in the tangential coordinates the mapping is approximated by the Optimal Transport mapping on $\Gamma$ and in the normal coordinates, the mapping simply changes by the normal derivative of the potential function divided by $\sigma$.

Now that we have an approximation of the Optimal Transport mapping, the remainder of the discretization is rather simple. We exploit the fact that in Euclidean space, for all the cost functions $c_{\sigma}(\mathbf{z}_1, \mathbf{z}_2)$ considered in this manuscript, a Euclidean derivative with respect to $\mathbf{z}_1$ is the same as a Euclidean derivative with respect to $\mathbf{z}_2$, with a minus sign.

In this way, we can fully discretize the PDE~\eqref{eq:OTExtension} on $\mathcal{T}_\epsilon^h$. For a point $\mathbf{z}^\prime \in T_{\epsilon}$, define $$U_{\mathbf{\mathbf{z}^\prime}}(\mathbf{z}) := v(\mathbf{z}) + c_{\sigma}(\mathbf{z}, \mathbf{z}^\prime),$$
and
\begin{equation}\label{eq:discretization}
    F^h_0(v(\mathbf{x}_i)) := \det D^2_h U_{\mathbf{m}(\mathbf{x}_{i})}(\mathbf{x}_i)- \left\vert \det D^2_{h} c_{\sigma} (\mathbf{x}_i, \mathbf{m}(\mathbf{x}_i)) \right\vert f(\mathbf{x}_i)/g(\mathbf{m}(\mathbf{x}_i)).
\end{equation}

The finite dimensional system of equations defined above is not closed, because parts of the finite difference stencil for those grid nodes close to $\partial T_\epsilon$ may lie outside of $T_\epsilon$. 
\begin{definition}
Let $\mathbf{x}_i\in h\mathbb{Z}^3$, 
$\mathcal{N}_h(\mathbf{x}_i)\subset h\mathbb{Z}^3$ denotes the set of grid nodes in the finite difference stencil used above.  
Points in 
$$ \mathcal{B}_h:= \{ \mathbf{y}_j\in \mathcal{N}_h(\mathbf{x}_i): \mathbf{x}_i\in T_{\epsilon}\cap h\mathbb{Z}^3,\mathbf{y}_j \notin T_\epsilon  \}$$
 are called boundary points. 
\end{definition}
A two-dimensional schematic showing how the computational grid is chosen, as well as a cross-section of computational points showing the boundary points and interior points, is shown in Figure~\ref{fig:grid}.


To obtained a closed system of equations, we will apply suitable closure to the nodes in $\mathcal{B}_h$. The closure condition need to be consistent to
the boundary condition in Equation~\eqref{eq:bc}. Since the solution $v$ of Equation~\eqref{eq:OTExtension} is \textit{a priori} constant in the normal direction, we elect to enforce the condition 
\begin{equation}\label{interp closure}
    v(\mathbf{y}_j):= \mathcal{I}_h[ v(P_\Gamma \mathbf{y}_j)], ~~~\mathbf{y}_j \in \mathcal{N}_{h}(\mathbf{x}_i)\cap \mathcal{B}_h.
\end{equation}



\begin{figure}
	\subfigure[]{\includegraphics[width=0.3\textwidth]{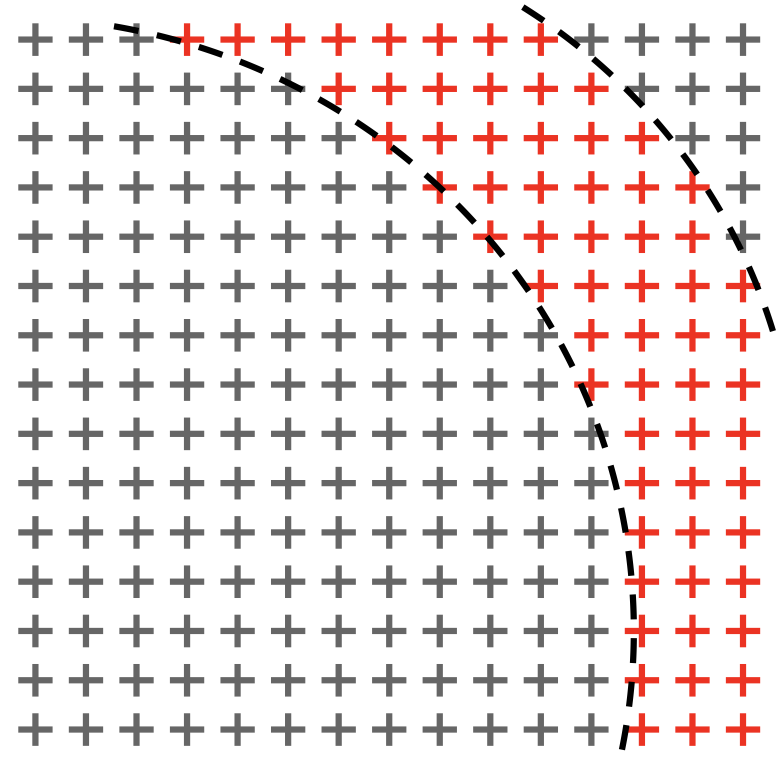}\label{fig:cloud}}
    \subfigure[]{\includegraphics[width=0.6\textwidth]{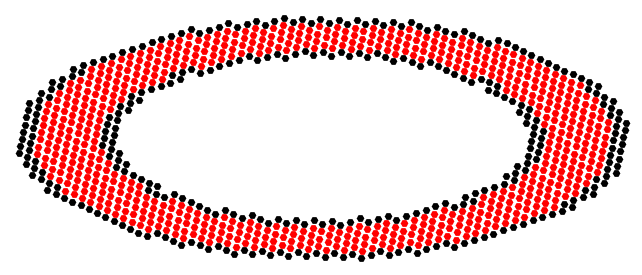}\label{fig:flat2}}
	\caption{
 (a) A schematic of the grid nodes in $\mathcal{T}^{h}_{\epsilon}$ (red); 
     (b) A horizontal cross-section of the computational grid nodes for 
     for the unit sphere,  
     showing the interior computational points 
     in red and the boundary points $\mathcal{B}^h$ in black.}\label{fig:grid}
\end{figure}

We choose to solve the discrete system~\eqref{eq:discretization}-\eqref{interp closure} via Algorithm~\ref{alg:solver}. It 
can be regarded as an accelerated gradient descent method as proposed in~\cite{SchaefferHou}, where we perform extrapolation with the choice of $\gamma(n) = n/(n+n_{0})$, for $n_0\geq 10$, as advocated. 
In practice, we observed that Algorithm~\ref{alg:solver}  without acceleration (setting $\gamma(n) = 0$) is much less efficient.

The iterations in Algorithm~\ref{alg:solver} terminate when the residual reaches a value below a desired tolerance. 
As can be seen from Equation~\eqref{eq:OTExtension}, the solution $v$ is unique up to a constant, since the PDE just depends on the derivatives of $v$. To settle the constant, once the iterations in either Algorithm~\ref{alg:solver} terminate, we find the minimum value of $u^h_K$ and define $v$ as $u^h_K$ minus this minimum value. Thus, the output of the computations is a grid function $v$ whose minimum value is zero.

\begin{algorithm}[h]
\caption{Jacobi-Type Iteration Update with Acceleration}
\label{alg:solver}
\begin{algorithmic}[1]
\State Given $v^h_0$ on $\mathcal{T}^h_\epsilon$, $\Delta t>0$, $\text{tol}>0$
\State $v^h_1(\mathbf{y}_j):= \mathcal{I}_h[ v^h_1(P_\Gamma \mathbf{y}_j)], ~~~\mathbf{y}_j \in \mathcal{B}_h$
\State $v^h_{1}(\mathbf{x}_i) := v^h_0(\mathbf{x}_i) + \Delta t F^h_0 \left( v^h_0(\mathbf{x}_i) \right),~\forall \mathbf{x}_i \in T_\epsilon\cap h\mathbb{Z}^3$ 
\State Set $n=1$
\While{$\max_{\mathbf{x}_i \in T_\epsilon\cap h\mathbb{Z}^3}\left\vert F^h_0 \left(v^h_n(\mathbf{x}_i) \right) \right\vert > \text{tol}$}
\State $\forall \mathbf{x}_i \in T_\epsilon\cap h\mathbb{Z}^3,$
\State $v^h_{n, E}(\mathbf{x}_i) := v^h_n(\mathbf{x}_i) + \gamma(n) (v^h_n(\mathbf{x}_i) - v^h_{n-1}(\mathbf{x}_i)))$, for all $\mathbf{x}_i \in T_\epsilon\cap h\mathbb{Z}^3$
\State $v^h_{n,E}(\mathbf{y}_j):= \mathcal{I}_h[ v^h_{n,E}(P_\Gamma \mathbf{y}_j)], ~~~\mathbf{y}_j \in \mathcal{B}_h$
\State $v^h_{n+1}(\mathbf{x}_i) := v^h_{n, E}(\mathbf{x}_i) + \Delta t F^h_0(v^h_{n,E}(\mathbf{x}_i))$
\State $n:=n+1$
\EndWhile
\State Define $v(\mathbf{x}_i) = v^h(\mathbf{x}_i) - \min_{i \in \mathcal{T}^h_\epsilon} v^h_{K}(\mathbf{x}_i)$.
\end{algorithmic}
\end{algorithm}

\begin{remarkun}

Monotone finite difference schemes are used for discretizing fully nonlinear elliptic PDEs to build convergence guarantees of the discrete solutions even in cases where the solution $v$ is only known to be \textit{a priori} continuous. There is a long line of work on the subject of monotone finite-difference discretizations of fully nonlinear second-order elliptic PDE, see~\cite{BSNum} for the theory showing uniform convergence of viscosity solutions of a monotone discretization of certain elliptic PDE, the paper~\cite{ObermanSINUM} for how the theory allows for the construction of wide-stencil schemes, the paper~\cite{HT_OTonSphere} for a convergence framework for building such monotone discretizations on local tangent planes of the sphere, and~\cite{HT_OTonSphere2} for an explicit construction of such a discretization. While it seems possible to construct a monotone scheme for the extended OT PDE~\eqref{eq:OTExtension}, we defer such development to a future project.
\end{remarkun}

\begin{remarkun}
    For the case of the sphere, $\Gamma = \mathbb{S}^2$, we may utilize known formulas for the mixed Hessian term $\left\vert \det D^2_{\mathbf{z} \boldsymbol{\xi}} c_{\sigma}(\mathbf{z}, \boldsymbol{\xi}) \right\vert$ derived in Section~\ref{sec:sphere}.
\end{remarkun}

\subsection{Computational Complexity}

The number of unknowns in the discrete system of equations for the Optimal Transport PDEs, \eqref{eq:OTExtension},
is asymptotically: 
\begin{equation*}
    |T_\epsilon\cap h\mathbb{Z}^3|\sim \mathcal{O}(\epsilon h^{-3}),~~~h\rightarrow 0.
\end{equation*}
If $\epsilon\sim\mathcal{O}(h^\alpha)$, for some $0\le\alpha\le 1$, the computational size of the discrete system would be
\begin{equation}
    \mathcal{O}(h^{\alpha -3}).
\end{equation}

When $\alpha=1$, the size of the system to be solved is similar to that of a discretized Optimal Transport PDEs defined in a bounded two dimensional domain. 

Formally, the Optimal Tranport PDEs in $T_\epsilon$ are derived for any $0<\epsilon<\tau_\Gamma$, independent of the 
mesh size $h$. 
This unique feature can be viewed both as an advantage as well and a disadvantage.

On the one hand, allowing $\epsilon \ge h$ brings flexibity in designing numerical algorithms, allowing, for example, wider stencil and more accurate interpolation scheme, or
finite element discretizations. On the other hand, larger $\epsilon$ implies a larger nonlinear system needs to be solved. 
In the numerical algorithm presented in this paper, the minimal value of $\epsilon$ is determined by the interpolation scheme. Since we employ trilinear interpolations, $\epsilon$ has to be larger than $\sqrt{3}h$. This means that formally, the computational complexity of solving the extended nonlinear system \eqref{eq:system} is roughly $2\sqrt{3}$ times higher than solving the two dimensional Optimal Transport PDEs.


Ultimately, one could raise the question about the convergence of the computed solution, 
especially as $\epsilon(h)\rightarrow 0$ as $h \rightarrow 0$. There are well-established theories on convergence of Trapezoidal rule based quadrature rules for approximating surface integrals under the proposed framework: see 
\cite{engquist2005discretization}\cite{zhong2023error}
for the regime $\epsilon\sim \mathcal{O}(h)$, and \cite{kublik2016integration, kublik2018extrapolative}\cite{izzo2022convergence,izzo2023high} for the regime $\epsilon \sim\mathcal{O}(h^\alpha),~0\le\alpha<1$.

Regarding the boundary value problem for the Optimal Transport PDE,
since the domain of depends on $\epsilon$, the convergence of the numerical solutions 
requires further analysis, and this line of analysis is not developed in this paper.

We notice that the PDEs \eqref{eq:OTExtension} do not formally depend on $\epsilon$: the $\epsilon$ factors in the ratio $f_{\epsilon}/g_{\epsilon}$ in \eqref{eq:OTExtension} cancel out. Therefore, there is no formal $\epsilon$ dependence in the finite difference scheme discussed earlier. However,
the closest point projection operator, $P_\Gamma$, implicitly depend on $\epsilon$ (recall that $P_\Gamma\mathbf{z}= (I-\phi_\Gamma(\mathbf{z})\nabla\phi_\Gamma)(\mathbf{z})$ and $|\phi_\Gamma(\mathbf{z})|<\epsilon$). Nevertheless, in practice, we observe that for a fixed $h$, $\epsilon=Ch^{\alpha}$, and $\epsilon=Ch$, the proposed discretization yields relatively accurate numerical solutions, see Section~\ref{sec:studies}. 

\subsection{Computational Results}

All computations in this section were performed using Matlab R2021b on a 2017 MacBook Pro, with a $2.3$ GHz Dual-Core Intel Core i$5$ and $16$ GB of $2133$ MHz LPDDR3 memory. 

In all of the computations, we used Algorithm~\ref{alg:solver} and initialized with the constant function $u^h_0 = 1$ and chose $\gamma(n) = (n+1)/(n+15)$ as per the recommendation in~\cite{SchaefferHou}. We perform a variety of computations on different surfaces, including the unit sphere, the northern hemisphere of the unit sphere, and a $2$-torus in $\mathbb{R}^3$. On the unit sphere, we use the squared geodesic cost function $c(\mathbf{x},\mathbf{y}) = \frac{1}{2}d_{\mathbb{S}^2}(\mathbf{x},\mathbf{y})^2$ and the logarithmic cost function $c(\mathbf{x}, \mathbf{y}) = -\log(1 - \mathbf{x} \cdot \mathbf{y})$ arising in the reflector antenna problem. For the reflector antenna problem, validation is typically done using a ray-tracing software to get a picture of the resulting target light illumination, see~\cite{castroreflector, point2fflens}. Since the main objective of this paper is to demonstrate the flexibility of our method over a variety of cost functions, we choose to perform a validation test for the logarithmic cost by comparing the computed solution to the solution of the PDE derived in Appendix~\ref{sec:exact} for a particular case. We supplement this with a numerical and qualitative validation of how well the pushforward constraint is satisfied. We also perform computations on the northern hemisphere of the unit sphere using the logarithmic cost. This example is done to show how the extension method can be used to perform computations on surfaces with boundaries, but also because it is a more physically relevant computation of the reflector antenna problem. Finally, we perform a computation on a torus $\Gamma \subset \mathbb{R}^3$, using the extended cost $c_{\sigma}(\mathbf{z}_1, \mathbf{z}_2) = \frac{1}{2} \left\Vert P_{\Gamma}\mathbf{z}_1 - P_{\Gamma}\mathbf{z}_2 \right\Vert^2 + \frac{\sigma}{2} (\phi(\mathbf{z}_1 - \mathbf{z}_2)^2$.

We point out that even though we have chosen to show the results of many of our computations by visualizing them on the base manifold $\Gamma$, all computations were performed by discretizing the extended Optimal Transport problem on $T_{\epsilon}$ as outlined in Section~\ref{sec:descriptionofscheme}.

\subsubsection{Comparison with the analytical solution on the sphere}\label{sec:analytical}

In Appendix~\ref{sec:exact}, for the sphere $\Gamma = \mathbb{S}^2$, we show the that if the source and target density satisfy:
\begin{align}\label{eq:knownsolution}
    f(x, y, z) &= \frac{\sin \left( \arccos(z) - \frac{\sin(\arccos(z))}{a_0}  \right)}{4 \pi \sin(\arccos(z))} \left(1 - \frac{z}{a_0} \right), \\
    g(x,y,z) &= \frac{1}{4 \pi},
\end{align}
for $a_0 = 3$, then if the cost function on $\mathbb{S}^2$ is the squared geodesic cost, then the Optimal Transport potential function equals:
\begin{equation}\label{eq:exact1}
    u(x,y,z) = \frac{z}{a_0}.
\end{equation}

Using $\epsilon = 0.2$, $h=0.05$ and $\sigma = 8$, we run Algorithm~\ref{alg:solver} and show that the residual and the $L^{\infty}$ error of the computed potential function with the known solution in Equation~\eqref{eq:exact1} converges. The $L^{\infty}$ error is computed by performing a trilinear interpolation of the computed potential function onto the sphere and then comparing it with the known analytical solution~\eqref{eq:exact1}. The results are summarized in Figure~\ref{fig:exactres}.
\begin{figure}
	\subfigure[]{\includegraphics[width=0.48\textwidth]{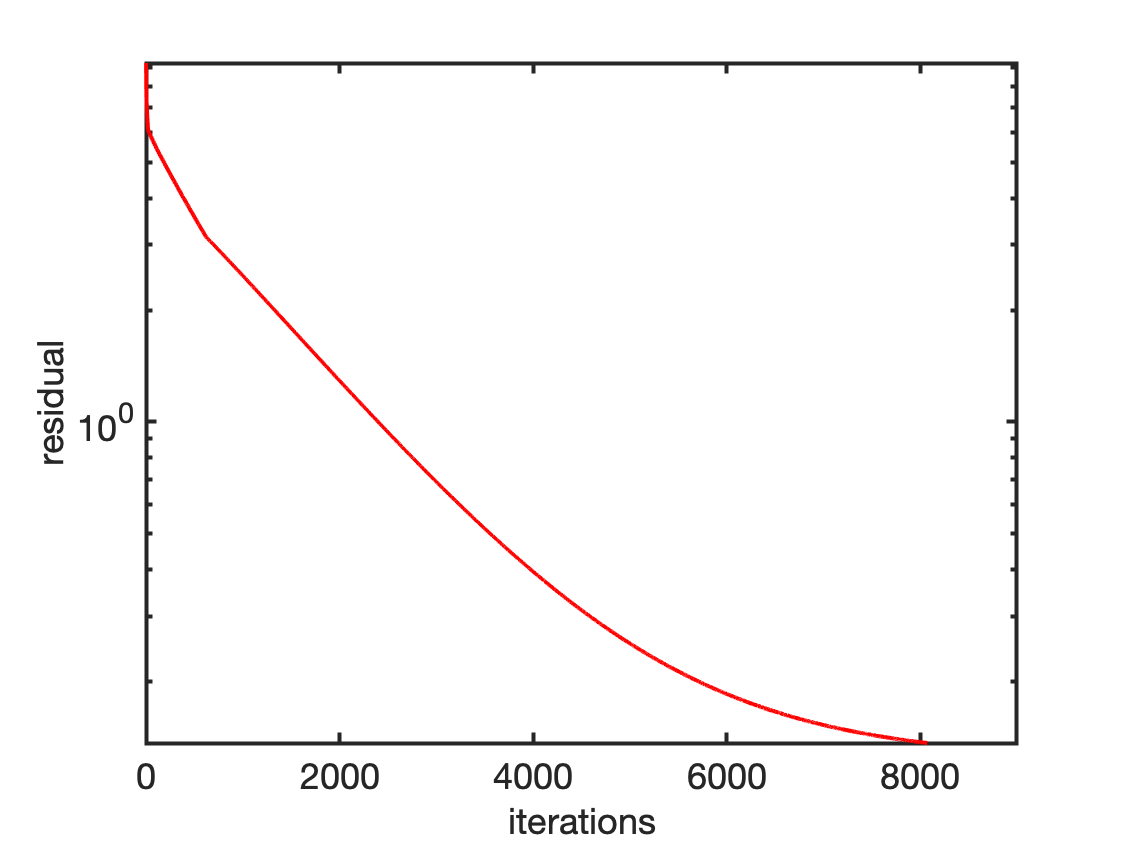}}
    \subfigure[]{\includegraphics[width=0.48\textwidth]{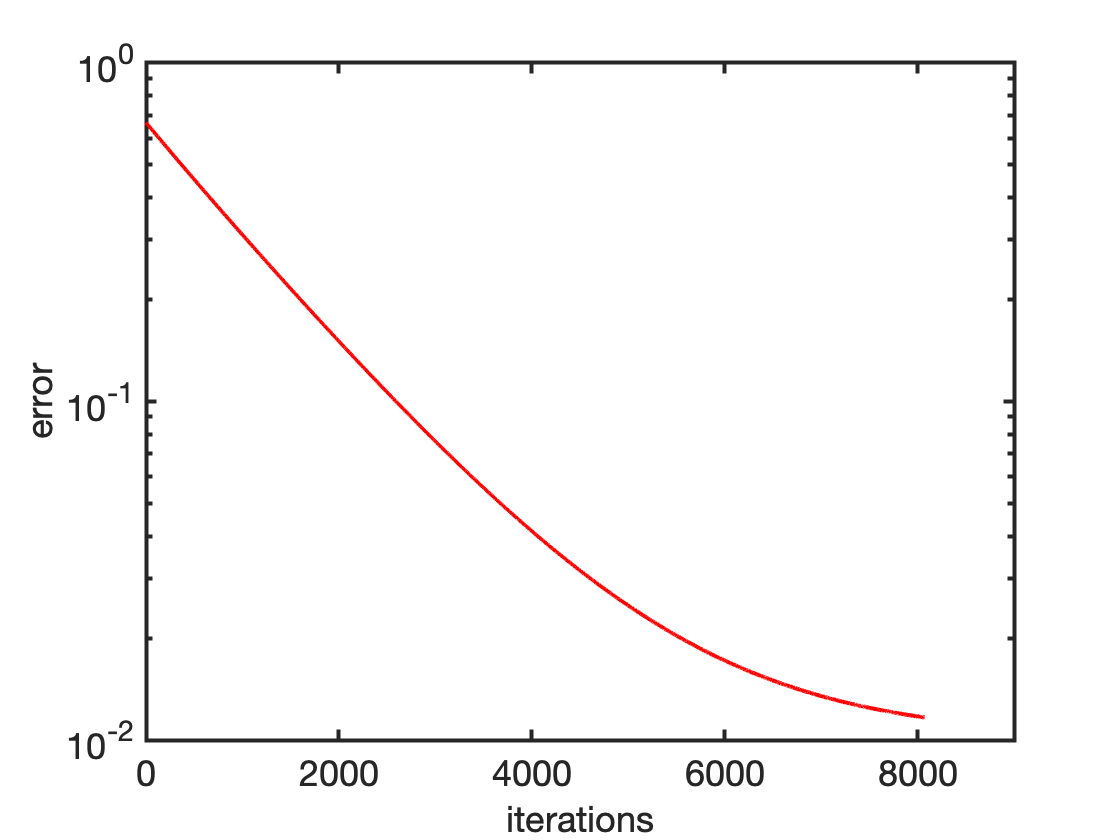}}
	\caption{(a) Convergence of the residual and (b) $L^{\infty}$ error.
 }\label{fig:exactres}
\end{figure}

We can, furthermore, test to see how well the pushforward constraint $\int_{\mathbb{S}^2}f(x)dx = \int_{\mathbb{S}^2}g(y)dy$ is satisfied with $a_0 = 2$. Since the mapping in this example is axisymmetric, we choose to fix a value $\theta_0$ and compute $T(\phi, \theta_0)$, which is approximately axisymmetric. For a fixed $\theta_0$ we then compare the values of the following integrals:
\begin{align}\label{eq:ibeforeafter}
    &\int_{0}^{2 \pi} \int_{0}^{\theta_0} f(\phi, \theta) d\theta d\phi, \\
    &\int_{0}^{2 \pi} \int_{0}^{T(\phi, \theta_0)} g(\phi, \theta) d\theta d\phi.
\end{align}

Since computational points on the grid $\mathcal{T}^h_\epsilon$ do not exactly equal $\theta_0$, we take a thin band around $\theta_0$, which becomes a thin band $T(\phi, \theta_0)$. The angle $\theta_0$ and $T(\phi, \theta_0)$ are then approximated by taking the average $\theta$-value in the respective bands. We can therefore approximate the integrals in Equation~\eqref{eq:ibeforeafter}. The result is summarized in Table~\ref{table:pushforward} and see that the pushforward constraint is approximately satisfied.
\begin{table}[h!]
\centering
\caption{Comparing $\int f(x)dx$ to $\int g(y)dy$ to see how well the pushforward constraint is satisfied.}
\label{table:pushforward}
\begin{tabular}{||c||c|c|c|c|c|c|c||} 
 \hline
 $\theta$ & $\pi/8$ & $2\pi/8$ & $3\pi/8$ & $4\pi/8$ & $5\pi/8$ & $6\pi/8$ & $7\pi/8$ \\ 
 \hline
  $\int f(x)dx$ & $0.0118$ & $0.0492$ & $0.1232$ & $0.2580$ & $0.4632$ & $0.7062$ & $0.9165$  \\ 
 \hline
  $\int g(y) dy$ & $0.0101$ & $0.0461$ & $0.1227$ & $0.2603$ & $0.4656$ & $0.7085$ & $0.9172$ \\
 \hline
\end{tabular}
\end{table}

We can also visualize the local change of area achieved by the computed pushforward map by seeing how a generated mesh is transformed via the pushforward map, see Figure~\ref{fig:pushforward}. Keep in mind that the source and target densities are given by Equation~\eqref{eq:knownsolution} and thus, the pushforward map should take the mesh and increase density around the north pole and decrease density around the south pole.

\begin{figure}
    \subfigure[]{\includegraphics[width=0.495\textwidth]
    {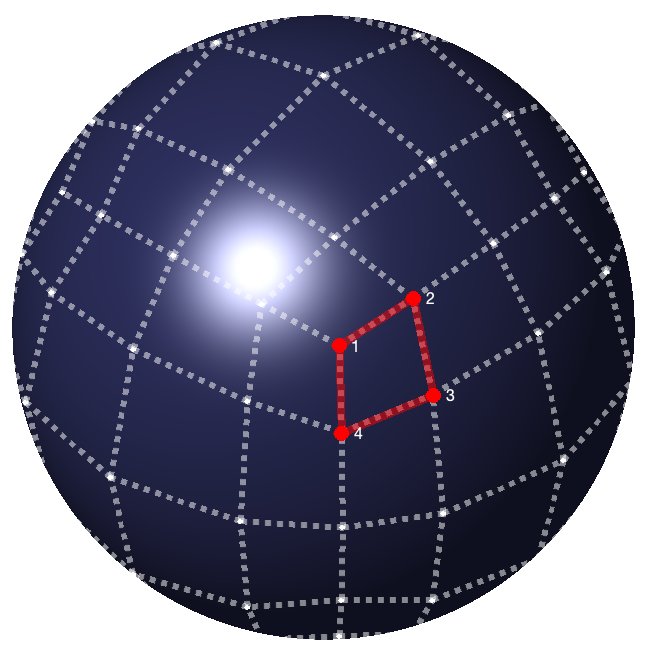}}
    \subfigure[]{\includegraphics[width=0.495\textwidth]{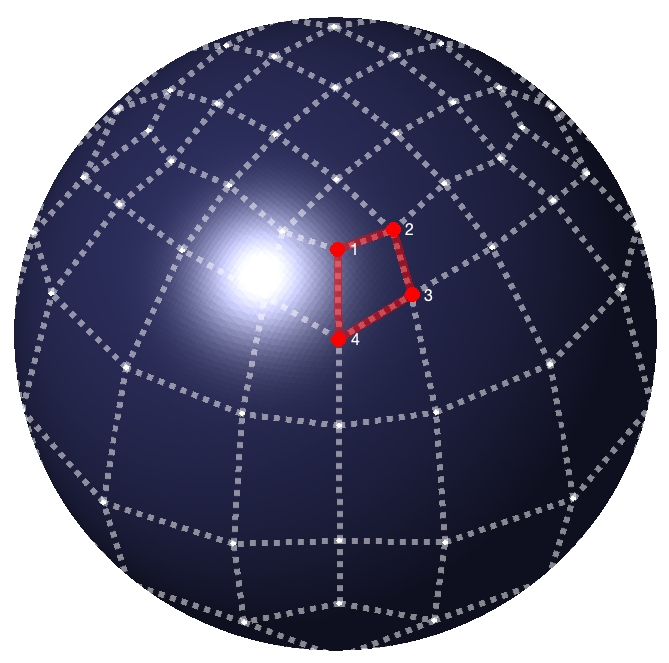}}    
	\caption{Visualization of the changes in the local density changes via a set of anchor points. The anchor points are connected to form a mesh over the sphere.
    (a) The initial distribution of anchor points; 
    (b) The distribution of the transported anchor points. 
    We highlight four connected vertices in red for comparison. The mass density in this computational example is required to increase in the northern hemisphere and decrease in the southern hemisphere. Details of this example can be found in Section~\ref{sec:analytical}}\label{fig:pushforward}
\end{figure}

\subsubsection{Studies with Varying $\sigma$, $\epsilon$, and $h$}\label{sec:studies}

In this section, we explore the effect of the ratio $\epsilon/h$ and the penalty parameter $\sigma$ using the test example from Section~\ref{sec:analytical}.

In our first study, we compare the error achieved after running Algorithm~\ref{alg:solver} for different values of $\epsilon/h$ with $\sigma=8$ and the source and target densities given by~\eqref{eq:knownsolution}. The runtime of the algorithm is, of course, much longer the higher the value of $\epsilon/h$, but the accuracy gets better as $\epsilon/h$ increases. The results are summarized in Table~\ref{table:epsilonoverh}.

\begin{table}[h!]
\centering
\caption{The $L^{\infty}$ error of the computed potential function to a known analytical solution decreases as a function of $\epsilon/h$.}
\label{table:epsilonoverh}
\begin{tabular}{||c||c|c|c|c|c||} 
 \hline
 $\epsilon/h$ & 2 & 3 & 4 & 5 \\ 
 \hline
 Error & $0.0625$ & $0.0123$ & $0.0102$ & $0.0089$ \\ 
  \hline
\end{tabular}
\end{table}



Next, we examine what happens when we vary $\sigma$, the penalty parameter, when $\epsilon=0.2$ and $h=0.05$, computing error using the analytical solution from Section~\ref{sec:analytical}. The results in Table~\ref{table:sigma} indicate that for ``small" values of $\sigma$, the main difficulty that the number of iterations required to reach a certain error is relatively large. For ``large" values of $\sigma$, the time step size has to be taken much smaller. Thus, we advocate for choosing $\sigma$ in the ``sweet spot" which is around $\sigma \approx 10$. This is why we often choose $\sigma=8$ in our computations. Table~\ref{table:sigma} summarizes the results of this study.

\begin{table}[h!]
\centering
\caption{Varying $\sigma$ and observing the number of iterations to reach a certain $L^{\infty}$ error. Note that the time step size must also change to avoid blowup of the algorithm.}
\label{table:sigma}
\begin{tabular}{||c||c|c|c|c|c||} 
 \hline
 $\sigma$ & $0.25$ & $2$ & $8$ & $16$ & $64$ \\ 
 \hline
 Error & 0.03798 & 0.01658 & 0.01165 & 0.01488 & 0.03185 \\ 
 \hline 
 Iterations & 27089 & 27945 & 8074 & 9257 & 8964 \\ 
 \hline
 $\Delta t$ & 0.0001 & 0.000025 & 0.000025 & 0.00001 & 0.0000025 \\  
  \hline
\end{tabular}
\end{table}

In our third study, we examine three different cases of the scaling as $\epsilon$ depending on $h$, to show convergence to the known solutions as $h \rightarrow 0$. In all cases, $\sigma=8$. In Table~\ref{table:tworootthree}, we set $\epsilon/h = 2\sqrt{3}$ and observe convergence in error as $h \rightarrow 0$. In this case, the scaling $\epsilon/h = \mathcal{O}(1)$, so the computational complexity scales as a two-dimensional discretization. We demonstrate that this scaling does well as $h \rightarrow 0$. The results are also summarized in Figure~\ref{fig:tworootthree}.

\begin{table}[h!]
\centering
\caption{Maintaining the ratio $\epsilon/h = 2\sqrt{3}$, we see convergence in the error.}
\label{table:tworootthree}
\begin{tabular}{||c||c|c|c|c|c||} 
 \hline
 $\epsilon$ & 0.25 & 0.2 & 0.15 & 0.1 & 0.07 \\ 
 \hline
 $h$ & $0.0686$ & $0.0577$ & 0.0433 & 0.0274 & 0.0192 \\ 
 \hline 
 Error & $0.0248$ & $0.0208$ & 0.0130 & 0.0074 & 0.0058 \\ 
 \hline
\end{tabular}
\end{table}

In Table~\ref{table:fixedepsilon}, we fix $\epsilon=0.2$ and vary $h$ and observe convergence of the error. In all cases $\sigma=8$. This method of fixing $\epsilon$ and shrinking $h$ scales as a true three-dimensional discretization. The case $\epsilon=0.2$ and $h=0.021$, for example involves $549322$ points in computation and, as such, is very slow. In this case, we can analytically prove the consistency of the numerical discretization, but caution its use in practice, do to the expense of computation. The results are also summarized in Figure~\ref{fig:tworootthree}.

\begin{table}[h!]
\centering
\caption{Fixing $\epsilon=0.2$ and shrinking $h$, we see convergence in the error.}
\label{table:fixedepsilon}
\begin{tabular}{||c||c|c|c|c|c||} 
 \hline
 $h$ & 0.079 & 0.061 & 0.045 & 0.031 & 0.021 \\ 
 \hline
 Error & 0.0270 & 0.0190 & 0.0129 & 0.0094 & 0.0076 \\ 
 \hline 
\end{tabular}
\end{table}

Next, we fix $\epsilon = Ch^{p}$ for $p=0.75$ and $C=1.6822$, which is an intermediate case of the scaling of the discretization, where the complexity scales between two-dimensional and three-dimensional methods. We set $\sigma=8$ in all test cases. The results are shown in Table~\ref{table:chp} and visualized in Figure~\ref{fig:tworootthree}.

\begin{table}[h!]
\centering
\caption{Setting $\epsilon=1.6822h^{0.75}$ and shrinking $h$, we see convergence in the error.}
\label{table:chp}
\begin{tabular}{||c||c|c|c|c|c||} 
 \hline
 $\epsilon$ & 0.25 & 0.2 & 0.15 & 0.1 & 0.08 \\ 
 \hline
 $h$ & $0.0787$ & $0.0585$ & 0.0398 & 0.0232 & 0.0172 \\ 
 \hline 
 Error & $0.0218$ & $0.0162$ & 0.0106 & 0.0060 & 0.0036 \\ 
  \hline
\end{tabular}
\end{table}

\begin{figure}
	\subfigure[]{\includegraphics[width=0.6\textwidth]{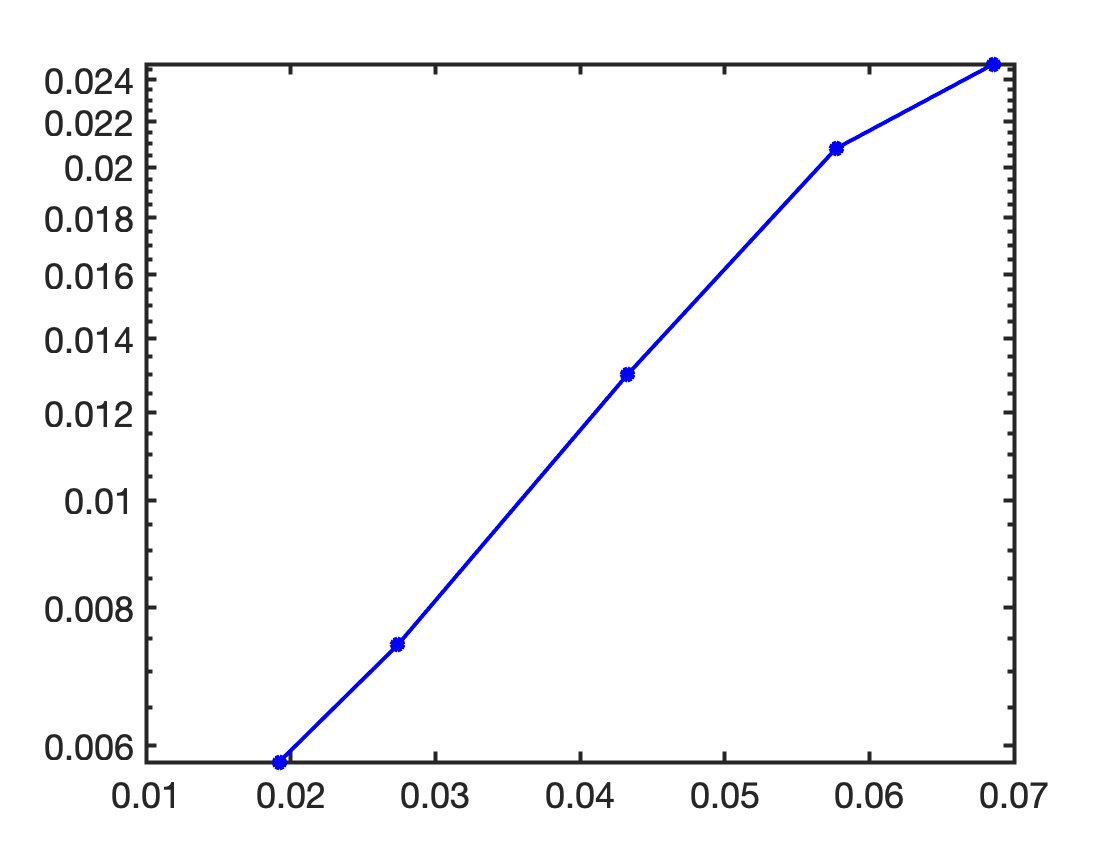}}
    \subfigure[]{\includegraphics[width=0.6\textwidth]{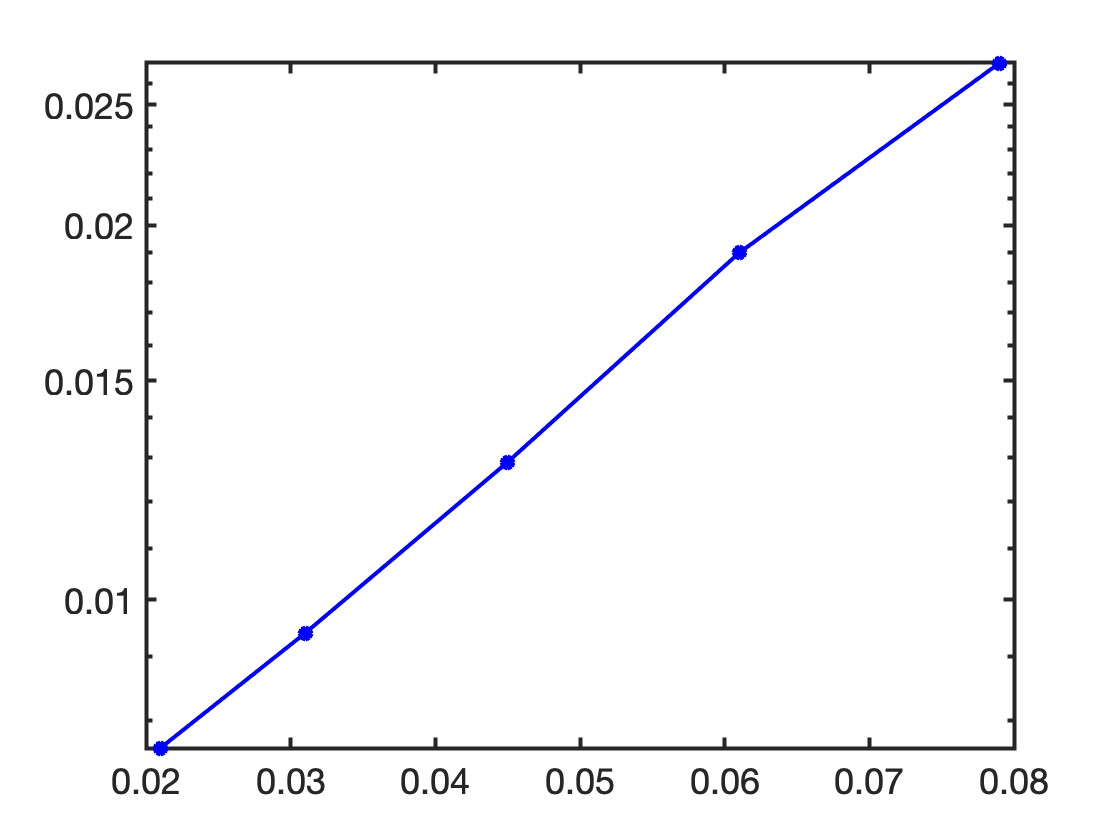}}
    \subfigure[]{\includegraphics[width=0.6\textwidth]{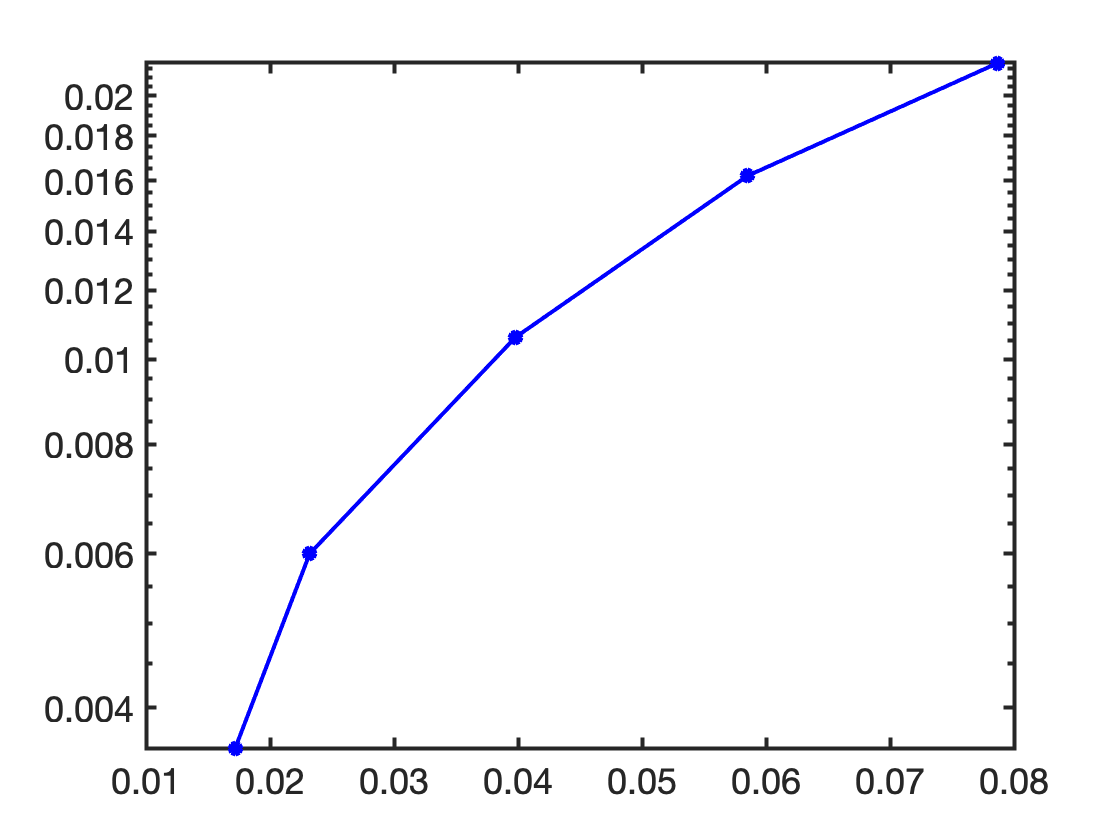}}
	\caption{The convergence of the error ($y$-axis) as a function of $h$ ($x$-axis) for (a) $\epsilon/h = 2 \sqrt{3}$, (b) for a fixed $\epsilon$, and (c) as a function of $h$ for $\epsilon=1.6822 h^{3/4}$.}\label{fig:tworootthree}
\end{figure}

\subsubsection{Optimal transport on a sphere: from North Pole to South Pole}
For this example, we have chosen $\epsilon=0.2$ and $h = 0.1$ and performed the computation on a grid of $5038$ points with $\sigma=1$. We performed the computation for squared geodesic cost function $c(\mathbf{x}, \mathbf{y}) = \frac{1}{2}d_{\mathbb{S}^2}(\mathbf{x}, \mathbf{y})^2$ and for the following source and target density functions for Cartesian coordinates $(x,y,z)$:
\begin{align*}
f(x,y,z) &= (1-\beta_1) \frac{1}{2\alpha_1} e^{-4(\arccos{z}-\frac{1}{10})^2} + \frac{\beta_1}{4 \pi}, \\
g(x,y,z) &= (1-\beta_2) \frac{1}{2\alpha_2} e^{-3(\arccos{z}-\pi + \frac{3}{10})^2} + \frac{\beta_2}{4 \pi},
\end{align*}
where $\beta_1 = 0.4$, $\beta_2 = 0.3$ $\alpha_1 = 1.042$, and $\alpha_2 = 2.089$, and whose extended densities are then derived by using Equation~\eqref{eq:densities}. Figure~\ref{fig:e2globe} shows the source and target densities on $\mathbb{S}^2$ and the computed potential function with a quiver plot projected to $\mathbb{S}^2$ showing the direction of the gradient of the potential function overlaid on top of a world map outline. We chose the world map to allow the reader to more easily visualize the location of the mass density concentrations. It should be clear from formulation of the Optimal Transport problem in Equation~\eqref{eq:OT} that in order to preserve mass the source mass located at the north pole needs to move to the target mass located at the south pole. In order to achieve this, the potential function must have a gradient near the north pole that points due south (towards the target mass distribution). Correspondingly, the direction of the mapping should point due south as well. This is precisely what we observe in Figure~\ref{fig:e2globe}.

Analytical and computational problems arise for solving the Optimal Transport PDEs when the target mass density $g$ approaches zero in the computation domain.  For one, the monotone discretizations on the sphere \cite{HT_OTonSphere2} requires that $g$ be bounded away from zero.
For this computational example, we have $h=0.1$ and we can take the parameter $\beta_1$ smaller and smaller, for example, down to at least $\beta_1 = 0.01$ without much of an effect on the number of iterations to reach a given tolerance. However, from numerical experiments, we have found that the Algorithm works fine above $\beta_2 =0.05 = h/2$, but not below this threshold, since the algorithm became unstable.




\begin{figure}
	\subfigure[]{\includegraphics[width=0.48\textwidth]{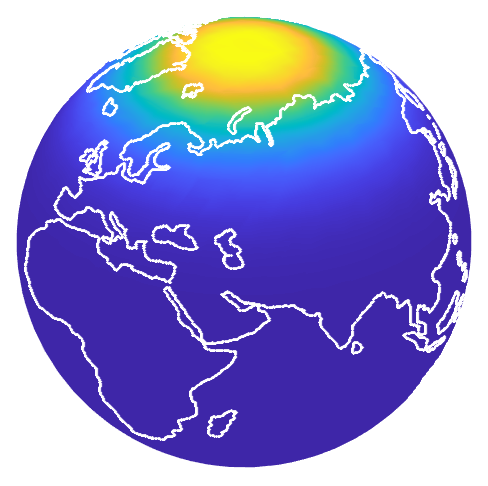}}
	\subfigure[]{\includegraphics[width=0.48\textwidth]{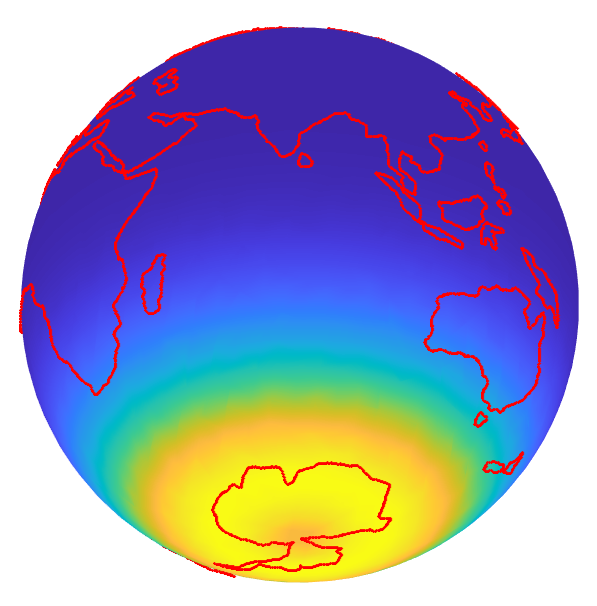}}
    \subfigure[]{\includegraphics[width=0.8\textwidth]{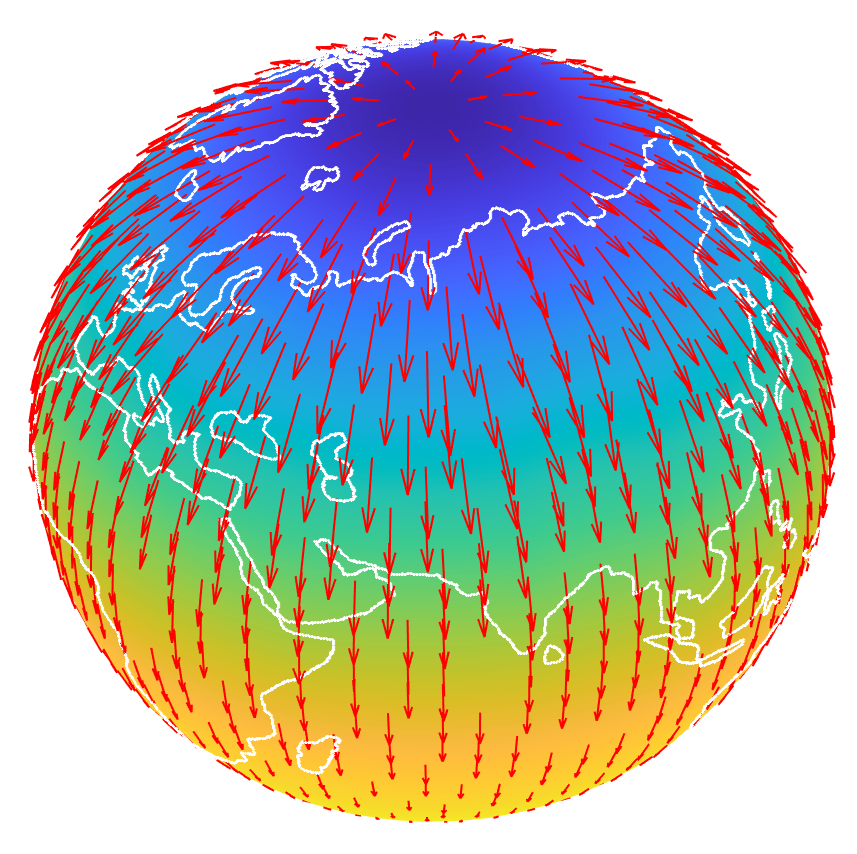}}
	\caption{(a) Source density at the north pole moves to (b) target density located at the south pole and (c) the resulting computed potential function with the direction of the gradient (red arrows) with an outline of a world map overlaid to enhance visualization.}\label{fig:e2globe}
\end{figure}


\subsubsection{Peanut Reflector}

In this example, we perform an Optimal Transport computation using the logarithmic cost $c(\mathbf{x}, \mathbf{y}) = -\log(1 - \mathbf{x} \cdot \mathbf{y})$ arising in the reflector antenna problem, with $\epsilon=0.2$ and $h=0.1$ and $\sigma=1$. In the reflector antenna problem, the computed potential function $u_{\epsilon}(\mathbf{z})$ is then used to construct the shape of a reflector $\rho$ via the expression:
\begin{equation}\label{eq:reflectorshape}
    \rho = P_{\mathbb{S}^2}(\mathbf{z})e^{-u_{\epsilon}(\mathbf{z})}, \ \ \ \mathbf{z} \in T_{\epsilon}.
\end{equation}

Note that the reflector is a sphere only when $u_{\epsilon}$ is constant. A constant potential function is the solution of the Optimal Transport problem on $\mathbb{S}^2$ only when $f = g$. In the reflector antenna problem, light originating at the origin (center of the sphere) with a given directional light intensity pattern $f(\mathbf{x})$ impinges off a reflector with shape given in Equation~\eqref{eq:reflectorshape}, and then travels in a direction $\mathbf{m}(\mathbf{x})$, yielding a far-field light intensity pattern $g(\mathbf{m}(\mathbf{x}))$. Conservation of light intensity, the physical laws of reflection, and a change of variables allow one to solve for the reflector shape via the function $u$ in Equation~\eqref{eq:reflectorshape} by solving for $u$ in the Optimal Transport PDE in Equations~\eqref{eq:OTGamma1} and~\eqref{eq:OTGamma2} with the logarithmic cost $c(\mathbf{x}, \mathbf{y}) = -\log(1 - \mathbf{x} \cdot \mathbf{y})$.

Strictly speaking, a reflector antenna should only be a hemisphere, since the intention is to redirect directional light intensity from the origin (inside the reflector) to the far-field outside the reflector. In order to do this, light must escape, and so the reflector cannot entirely envelop the origin. Temporarily putting physical realities aside, we compute the shape of the reflector for light intensity functions $f$ and $g$ with support equal to $\mathbb{S}^2$, demonstrating that the computation can be done for the entire sphere for the cost function $c(\mathbf{x}, \mathbf{y}) = -\log(1-\mathbf{x} \cdot \mathbf{y})$. We use a source density function which resembles a headlights of a car projected onto a sphere and a constant target density. The densities and the resulting ``peanut-shaped" reflector is shown in Figure~\ref{fig:reflector}. This computation was inspired by the work in~\cite{RomijnSphere} and was also done with a different discretization in~\cite{HT_OTonSphere3}.

\begin{figure}
	\subfigure[]{\includegraphics[width=0.48\textwidth]{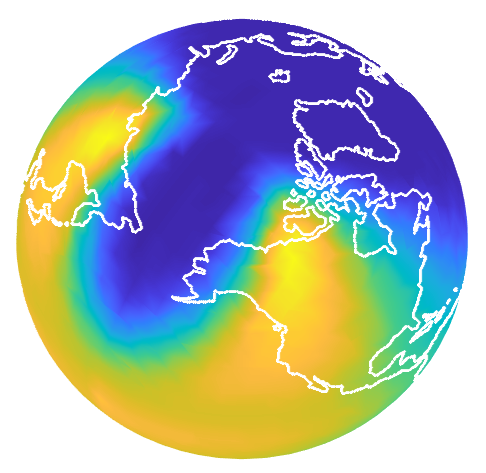}\label{fig:peanutlensf}}
	\subfigure[]{\includegraphics[width=0.48\textwidth]{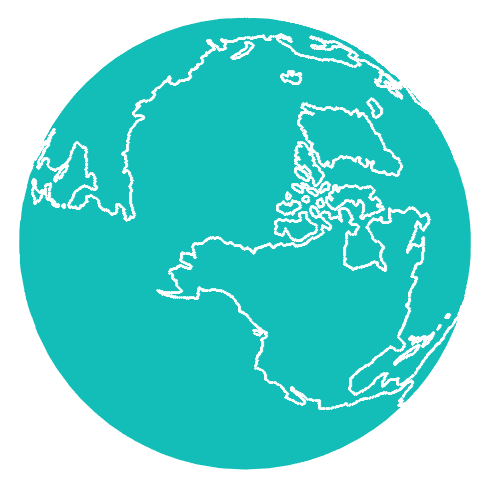}\label{fig:gconstant}}
    \subfigure[]{\includegraphics[width=0.8\textwidth]{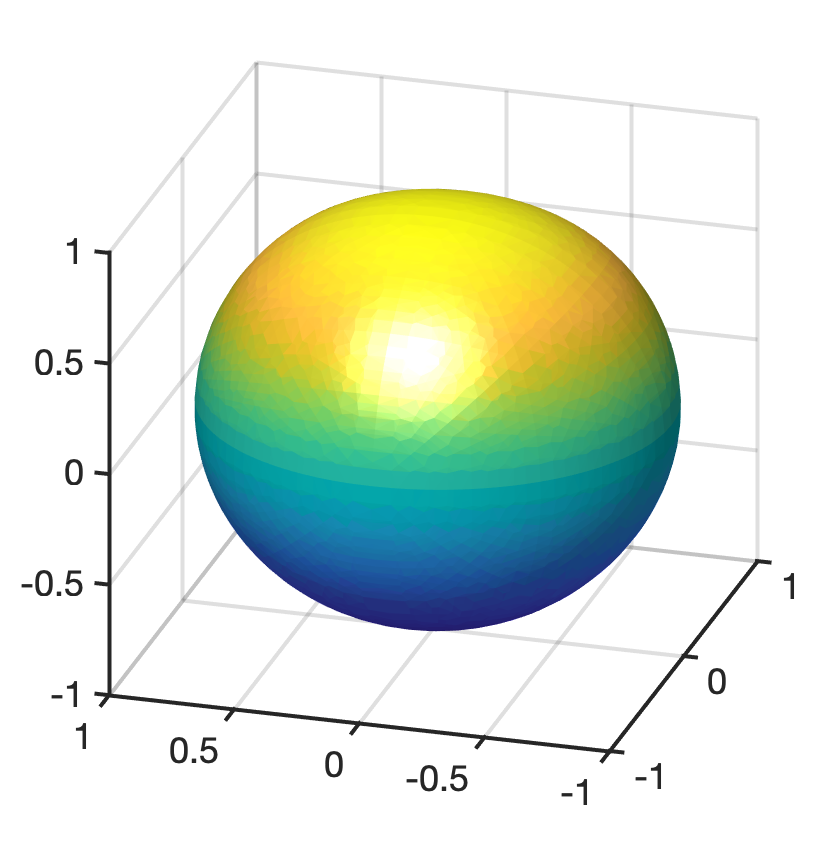}\label{fig:e3reflector2}}
	\caption{(a) Source density; (b) target uniform density; (c) the resulting reflector shape. Note that the reflector is not a sphere, but rather the shape given in Equation~\eqref{eq:reflectorshape}.}\label{fig:reflector}
\end{figure}

We perform another computation using the logarithmic cost function in order to compare with an exact solution. According to Appendix~\ref{sec:exact}, if the source and target densities satisfy:
\begin{align}
    f(x,y,z) &= \frac{\sin \left( \pi - \arccos(z) - \arccos \left( \frac{1-z^2-a_0^2}{1-z^2+a_0^2} \right) \right)}{4 \pi \sqrt{1-z^2}} \left(-1 + \frac{2a_0 z}{1-z^2+a_0^2} \right), \\
    g(x,y,z) &= \frac{1}{4 \pi},
\end{align}
then the potential function is known analytically to be:
\begin{equation}
u(x,y,z) = \frac{z}{a_0}.
\end{equation}

We perform our computation using Algorithm~\ref{alg:solver}, with $\epsilon=0.2$, $h=0.05$ and $\sigma=8$ and the results are shown in Figure~\ref{fig:reflectorantenna}.

\begin{figure}
	\subfigure[]{\includegraphics[width=0.48\textwidth]{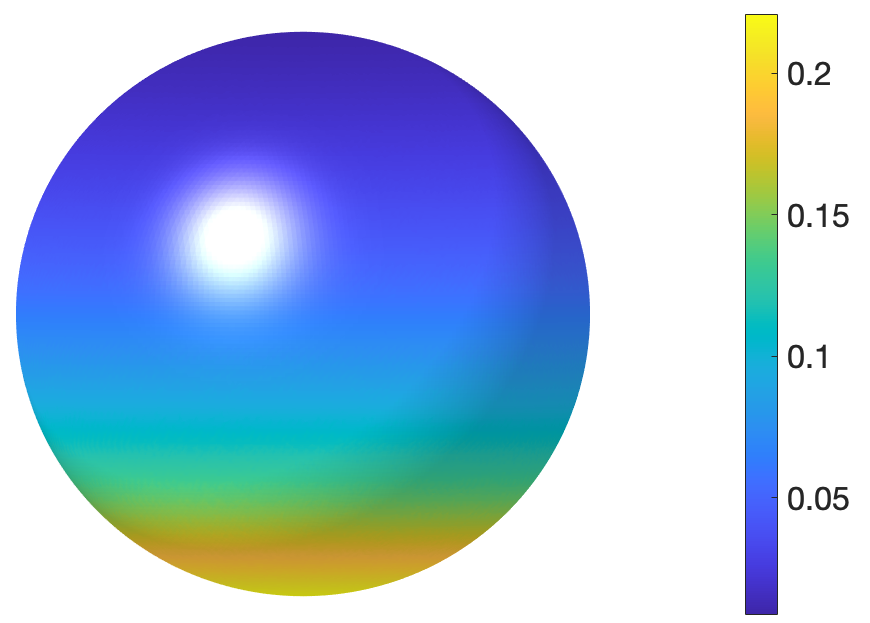}\label{fig:reflectorf}}
	\subfigure[]{\includegraphics[width=0.46\textwidth]{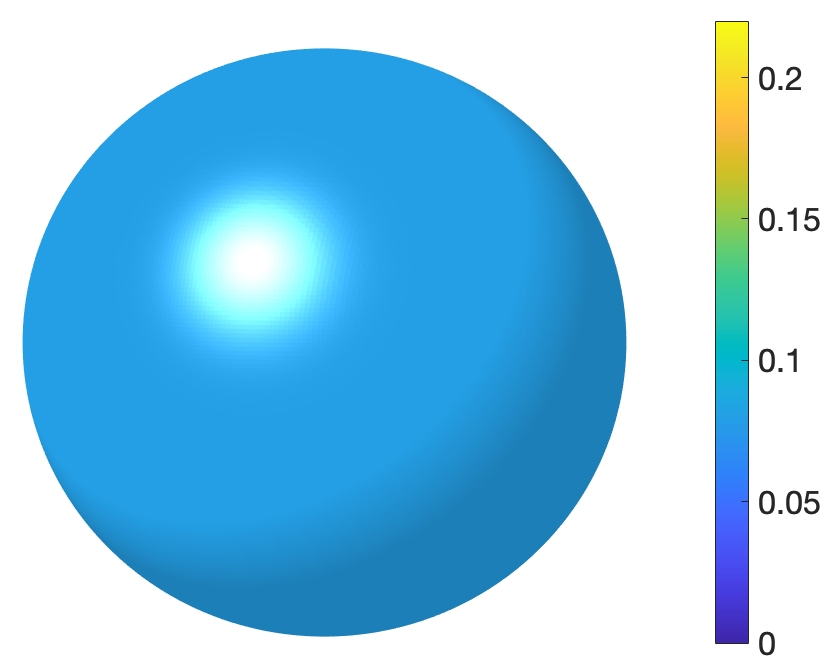}\label{fig:reflectorg}}
    \subfigure[]{\includegraphics[width=0.48\textwidth]{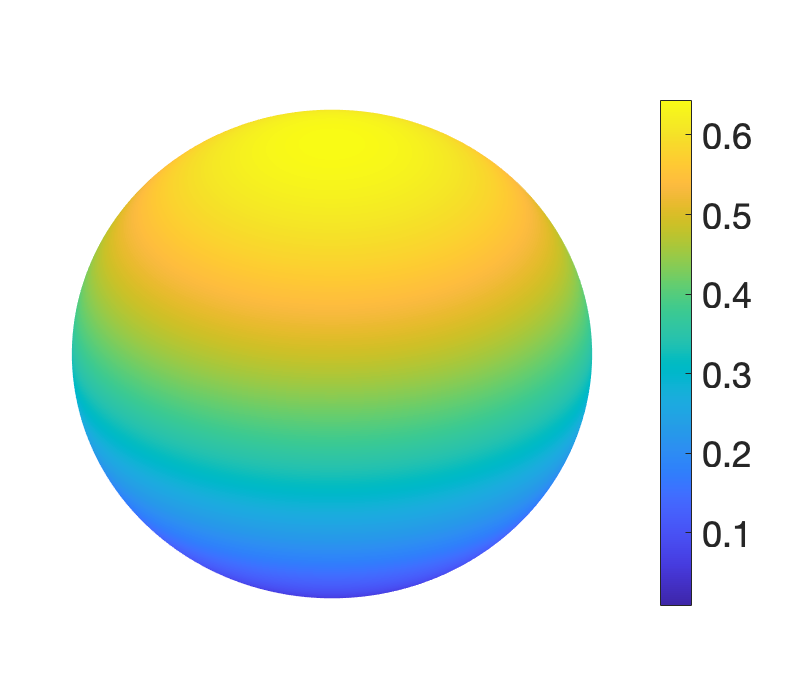}\label{fig:reflectoru}}
    \subfigure[]{\includegraphics[width=0.48\textwidth]{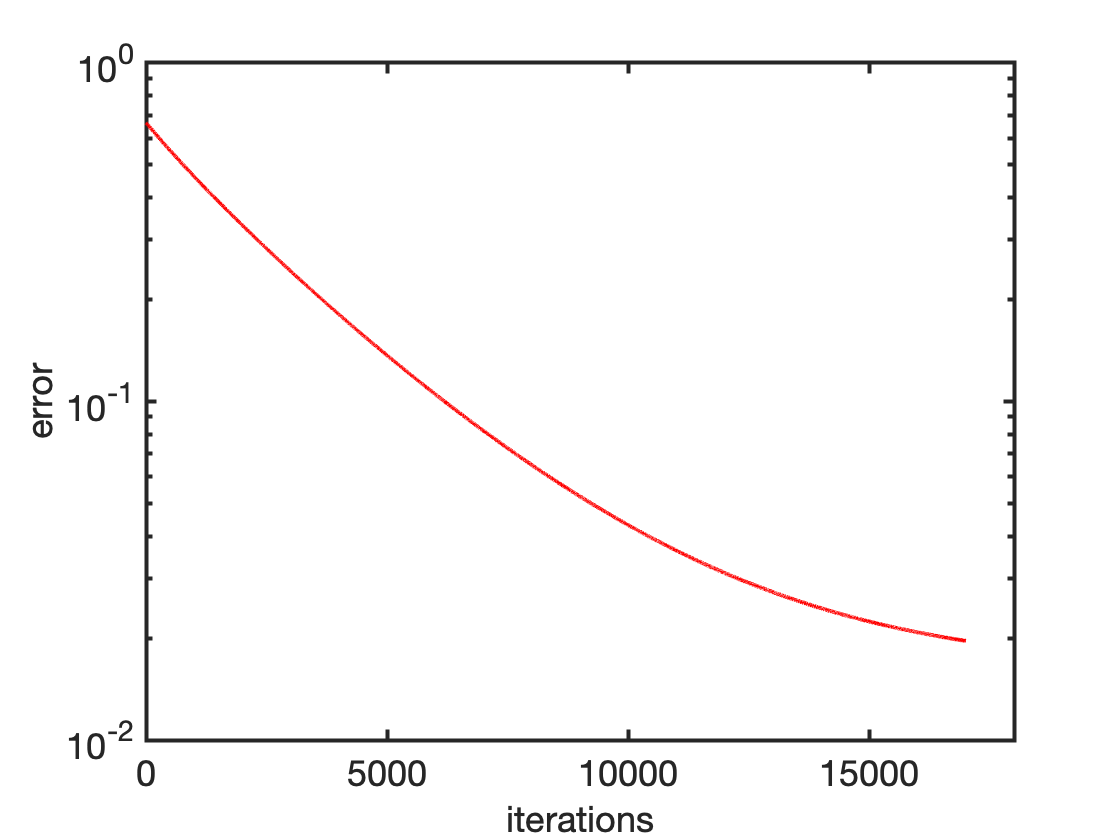}\label{fig:reflectorerror}}
	\caption{(a) Side view of the source density, (b) side view of the target density, (c) computed potential function and (d)  convergence of error.}\label{fig:reflectorantenna}
\end{figure}

\subsubsection{Non-Lipschitz target density}

For this example, we have chosen $\epsilon=0.2$ and $h = 0.1$ performed the computation on a grid of $5014$ points with $\sigma=1$. We perform the computation for the squared geodesic cost function $c(\mathbf{x}, \mathbf{y}) = \frac{1}{2}d_{\mathbb{S}^2}(\mathbf{x}, \mathbf{y})^2$ and for the following source and target density functions for Cartesian coordinates $(x,y,z)$: 
\begin{align*}
f(x,y,z) &= (1-\beta_1) \frac{1}{\alpha_1} e^{-5\arccos{x}^2} + \frac{\beta_1}{4 \pi}, \\
g(x,y,z) &= (1-\beta_2)/2\pi, & \text{if} \ z \cdot \mathbf{n} \geq 0 \\
g(x,y,z) &= \beta_2/4\pi, & \text{otherwise},
\end{align*}
where $\mathbf{n} = (0.2, 0.2, 1)$, $\beta_1 = 0.3$, $\beta_2 = 0.1$ and $\alpha_1 = 0.607788$ and the extended density functions are given by Equation~\eqref{eq:densities}. The target mass density in this example is discontinuous and the set of discontinuity has been chosen to not align with the computational grid $\mathcal{T}^h_\epsilon$. This is an important and difficult test example since the target density, while bounded away from zero, is not Lipschitz. Figure~\ref{fig:tilted} shows the source and target densities and the resulting quiver plot of the direction of the gradient of the potential function. Since we are performing a computation with the squared geodesic cost function, the gradient of the potential function around where the source mass is concentrated (middle of the Pacific Ocean) should be pointing to the northeast, but fanning out, so the mass gets appropriately spread out. This is, in fact, what we see from Figure~\ref{fig:tilted}.

\begin{figure}
	\subfigure[]{\includegraphics[width=6cm]{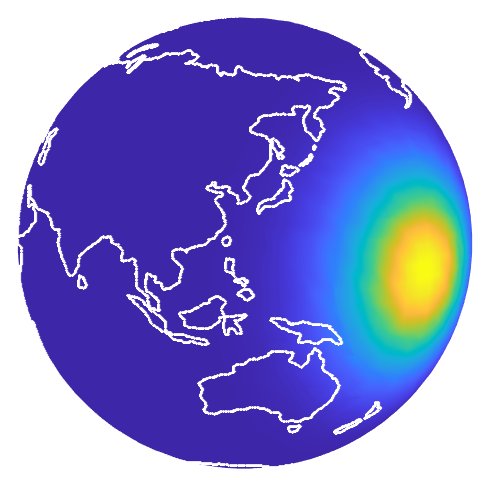}\label{fig:tiltf}}
	\subfigure[]{\includegraphics[width=6cm]{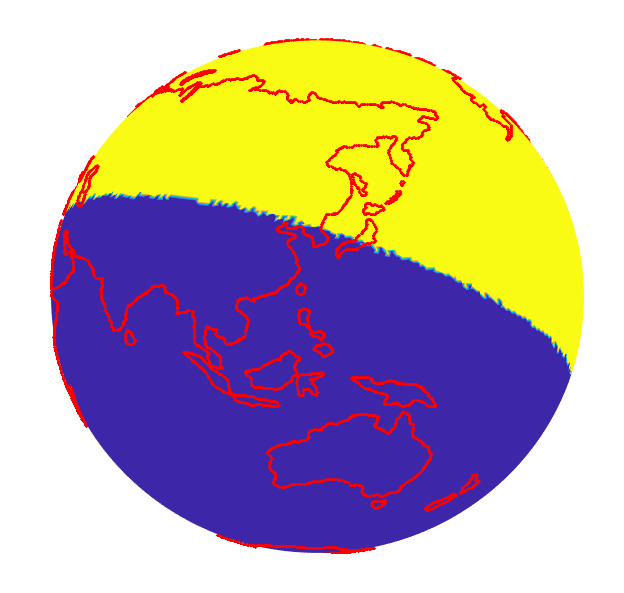}\label{fig:tiltg}}
    \subfigure[]{\includegraphics[width=0.8\textwidth]{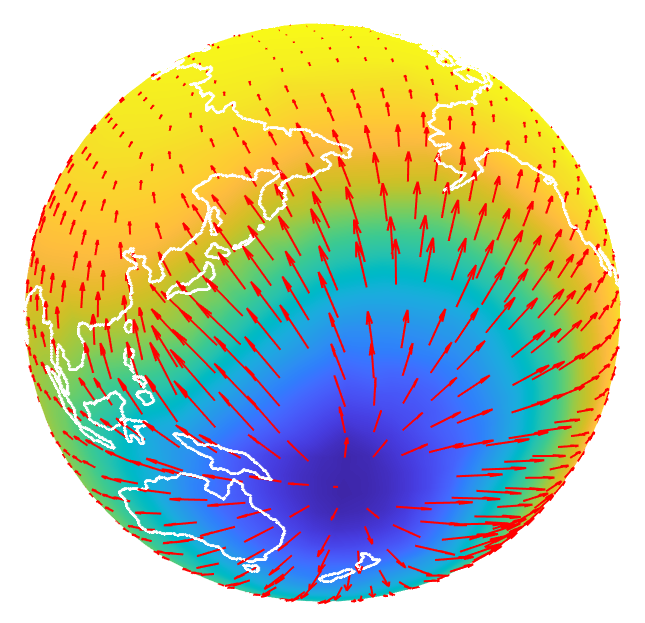}\label{fig:tiltmap}}
	\caption{ (a) Source density, (b) target density and (c) the potential function with its gradient (red arrows) showing mass moving from the source density in the middle of the Pacific Ocean to the half of the globe.}\label{fig:tilted}
\end{figure}

\subsubsection{Reflector antenna design using a hemisphere}

As outlined in Section~\ref{sec:boundary}, we can use a modification of Algorithm~\ref{alg:solver} to solve Optimal Transport on the northern hemisphere of $\mathbb{S}^2$, which is a surface with boundary, provided that we augment the discretization with zero-Neumann boundary conditions on the extension of the boundary. We use the density functions:
\begin{align}
    f(x,y,z) &= \frac{1+10 \arccos(z)}{2 \pi(1+a_0)}, \ \ \ &z \geq 0, \\
    g(x,y,z) &= \frac{1}{2 \pi}, \ \ \ &z \geq 0.
\end{align}

For this computation, we use $\sigma=8$, $\epsilon=0.2$ and $h=0.05$. The resulting grid has $N=24208$ points. We perform this computation using the logarithmic cost function $c(x,y) = -\log(1-x \cdot y)$, since this corresponds to a more physically accurate computation of the reflector antenna problem, as the computation is only for a light intensity pattern on one hemisphere. Even though, strictly speaking, a reflector antenna computation must be done over the entire sphere, we model a source light intensity pattern in the northern hemisphere being reflected to a target light intensity pattern in the southern hemisphere via a computation using the logarithmic cost function $c(x,y) = -\log(1-x \cdot y)$ only in the northern hemisphere.

The results of the computation are shown in Figure~\ref{fig:boundarycomputation}, which has a gradient of approximately zero along the boundary (strictly speaking the gradient is zero for the ghost points $\mathcal{C}^h$ which are located a distance $h$ below the boundary). The source mass, which is concentrated more around the equator, moves toward the north pole, as expected.

\begin{figure}
	\subfigure[]{\includegraphics[width=0.495\textwidth]{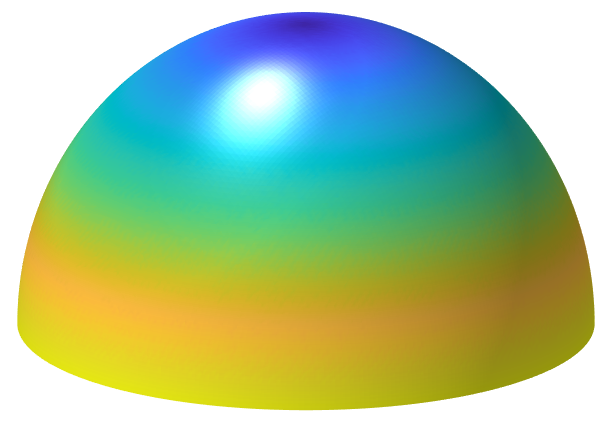}}
    \subfigure[]{\includegraphics[width=0.49\textwidth]{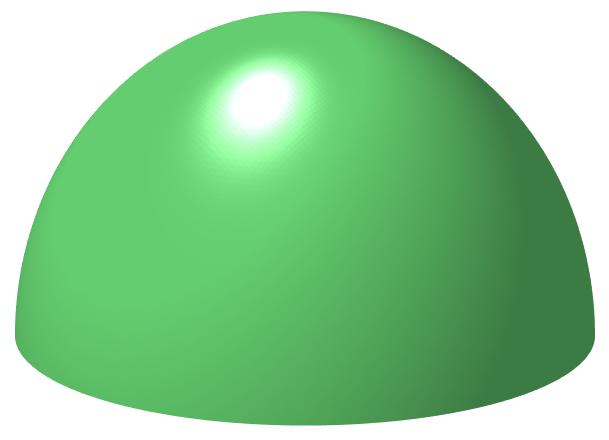}}
    \subfigure[]{\includegraphics[width=0.8\textwidth]{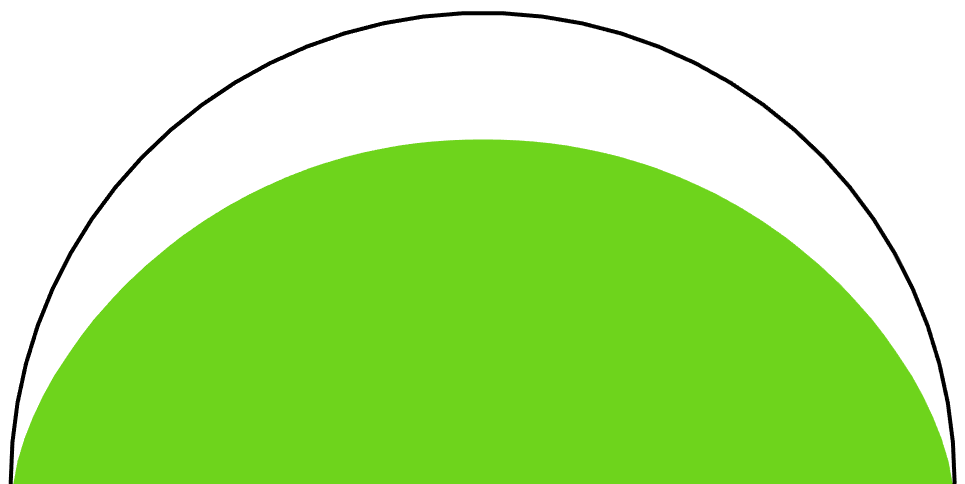}}
	\caption{(a) Source mass, which is concentrated near the boundary gets sent to the (b) target mass, which is constant over the hemisphere. (c) The reflector, side view, which resembles a flattened hemisphere, is able to redistribute the brighter light intensity along the equator to a constant light intensity in the southern hemisphere. A true hemisphere (which would lead to an Optimal Transport mapping that is the identity map) is shown with a black line. }\label{fig:boundarycomputation}
\end{figure}

\subsubsection{Optimal Transport computation on a torus}\label{sec:torus}
We consider the torus parametrized as follows:
\begin{align}\label{eq:torusparametrization}
    x &= \left(c+a \cos v \right) \cos u, \\
    y &= \left(c+a \cos v \right) \sin u, \\
    z &= a \sin u,
\end{align}
for $u, v \in [0, 2\pi)$ and $a=0.65$ and $c=1.3$. 

The signed distance function $\phi_{\Gamma}$ is given by the explicit formula:
\begin{equation}
    \phi_{\Gamma}(\mathbf{z}) = \sqrt{\left(\sqrt{x^2+y^2}-c \right)^2 + z^2}-a,
\end{equation}
and the projection operator $P_{\Gamma}$ is defined by the following map:
\begin{equation}
(x, y, z) \mapsto \left( \left( c + \frac{a}{\rho}(r-c) \right) \frac{x}{r}, \left( c + \frac{a}{\rho} (r-c) \right) \frac{y}{r}, \frac{a}{\rho}z \right).
\end{equation}

We apply Algorithm~\ref{alg:solver}, with $\epsilon=0.2$, $h=0.05$, and $\sigma=8$, to solve the resulting discrete equations. The resulting grid has $N=106304$ points.

For cost functions that do not require pre-computation, the runtime of Algorithm~\ref{alg:solver} is similar for that on the sphere $\mathbb{S}^2$. We will use the following cost function, since it does not require the expensive computation of geodesics on $\Gamma$:
\begin{equation}
c_{\sigma}(\mathbf{z}_1, \mathbf{z}_2) = \frac{1}{2} \left\Vert P_{\Gamma} (\mathbf{z}_1)-P_{\Gamma}(\mathbf{z}_2) \right\Vert^2+ \frac{\sigma}{2} \left( \phi_{\Gamma}(\mathbf{z}_1) - \phi_{\Gamma}(\mathbf{z}_2) \right)^2.
\end{equation}

On a torus, given an appropriate potential function $u$, the mapping for such a cost function satisfies:
\begin{equation}
    \mathbf{m}(\mathbf{x}) = \mathbf{x} + \nabla u(\mathbf{x}) + \alpha \mathbf{n}(\mathbf{x}),
\end{equation}
where $\alpha$ is the smallest value for which $\mathbf{m}(\mathbf{x}) \in \Gamma$. As this shows, the problem with such a cost function, though cheap, is that the source and target masses cannot require the mapping to transport mass too far, see~\cite{TurnquistLens} for more examples of cost functions that do not allow mass to be transported beyond a certain geodesic distance. 

The source and target densities we use for this computation are:
\begin{align}
    f(x,y,z) &= \frac{z+a+a_0}{4 \pi^2 a c (a+a_0)}, \\
    g(x,y,z) &= \frac{1}{4 \pi^2 a c},
\end{align}
for $(x,y,z) \in \Gamma$, where $a_0 = 0.5$. The source and target masses and computed potential function are shown in Figure~\ref{fig:torus}.

\begin{figure}
	\subfigure[]{\includegraphics[width=0.48\textwidth]{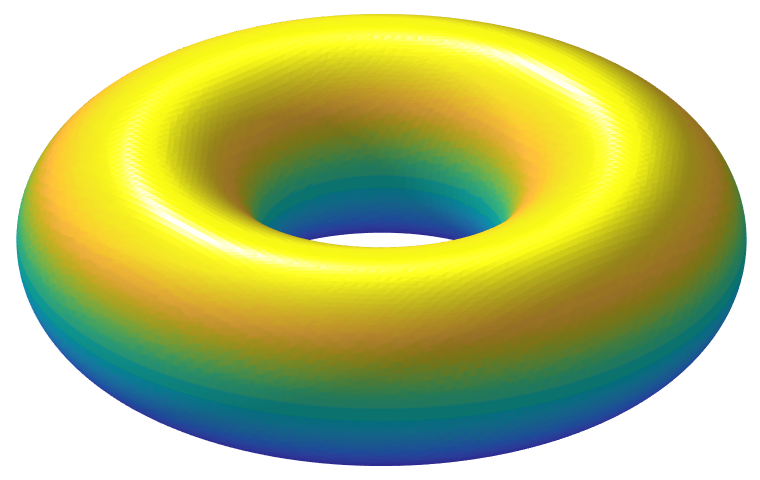}}\label{fig:torusf2}
    \subfigure[]{\includegraphics[width=0.495\textwidth]{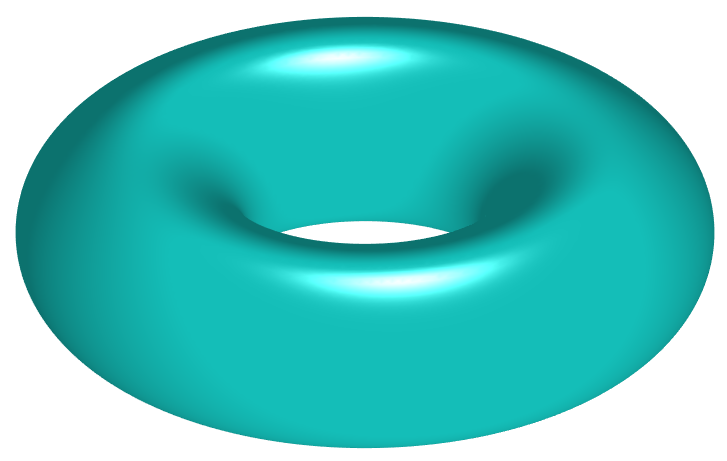}}\label{fig:torusg}
    \subfigure[]{\includegraphics[width=0.8\textwidth]{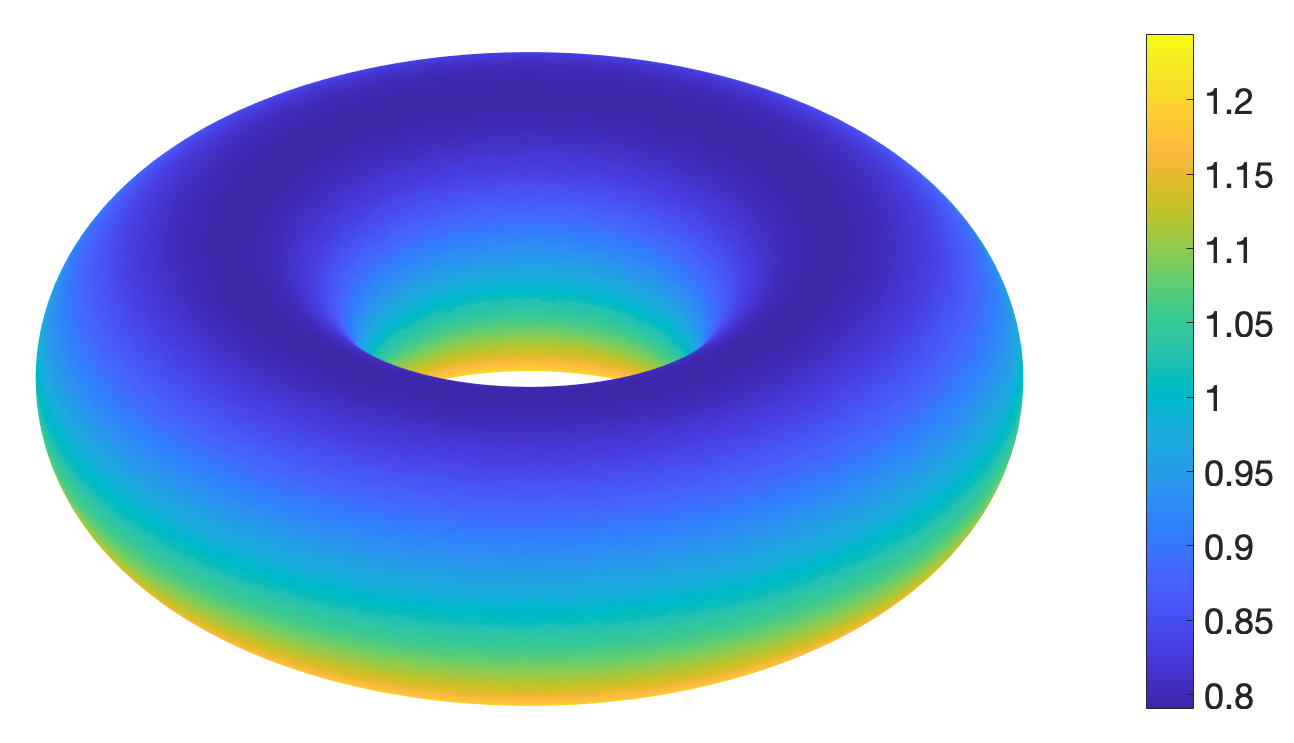}}\label{fig:torusu2}
	\caption{(a) Source mass, which is concentrated near the top of the torus gets sent to (b) the target mass, which is constant  over the entire torus, and (c) the computed potential function.}\label{fig:torus}
\end{figure}

The use of this cost function may be an efficient expedient for moving mesh methods on the torus, since any cost function used for Optimal Transport computations that yields a pushforward mapping is valid to use, see~\cite{Turnquist} for an overview of moving mesh methods using Optimal Transport. We implement a moving mesh computation on the torus for $a_0=1$, where transported mesh is not necessarily symmetric to the grid due to the computational grid being generated differently than the mesh. The results are shown in Figure~\ref{fig:torus2}, which demonstrate that the density of mesh points on the top of the torus is decreasing, while the density of mesh points on the bottom of the torus is increasing. Also, the mesh does not exhibit tangling while keeping the same connectivity, another beneficial aspect of using Optimal Transport for moving mesh methods.

\begin{figure}
	\subfigure[]{\includegraphics[width=0.48\textwidth]{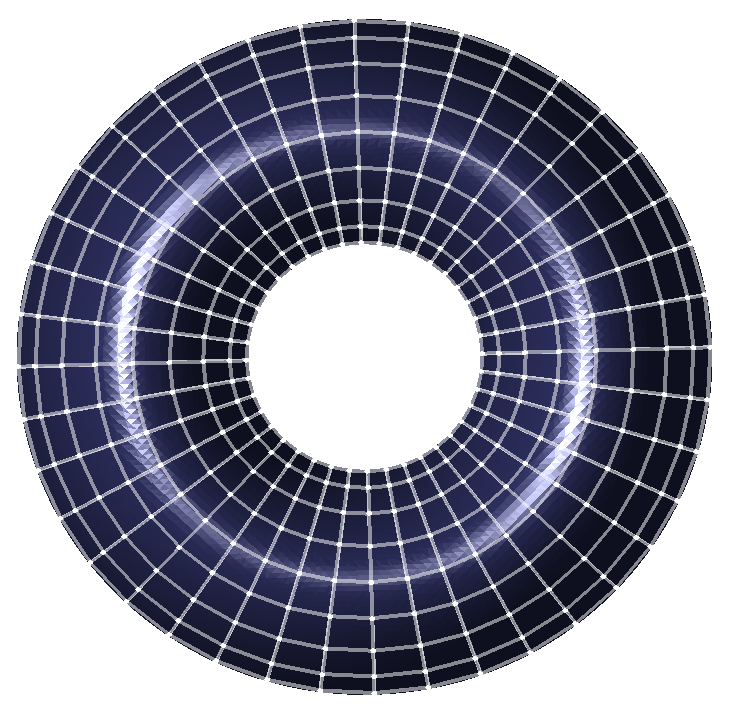}}\label{fig:beforetop}
    \subfigure[]{\includegraphics[width=0.495\textwidth]{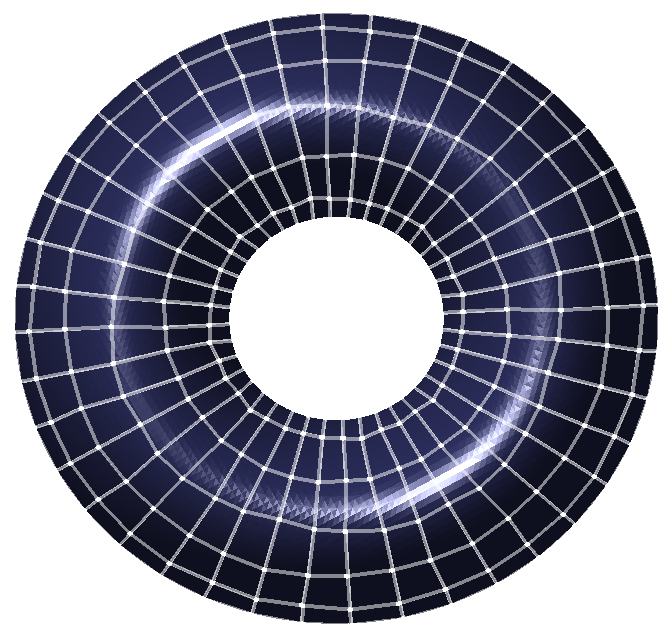}}\label{fig:aftertop}
    \subfigure[]{\includegraphics[width=0.48\textwidth]{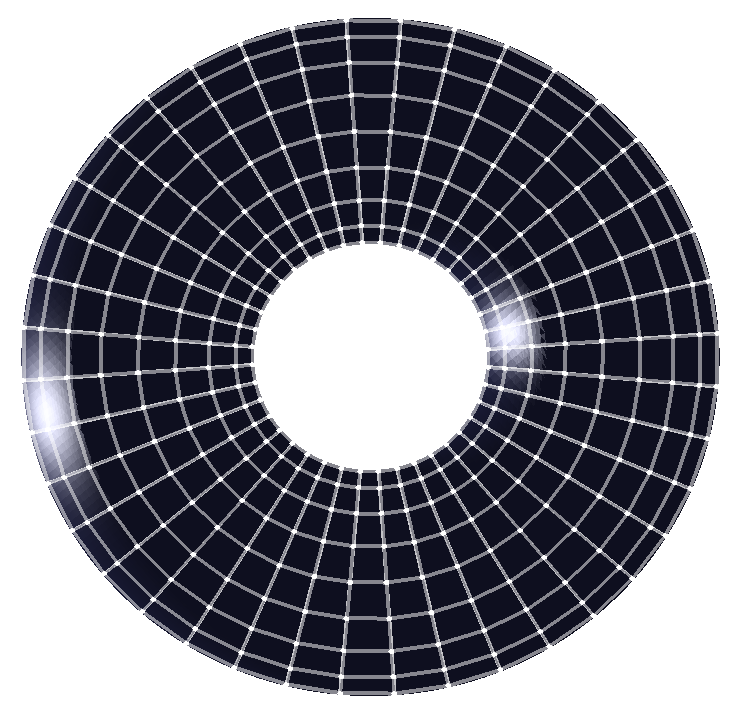}}\label{fig:beforebottom}
    \subfigure[]{\includegraphics[width=0.495\textwidth]{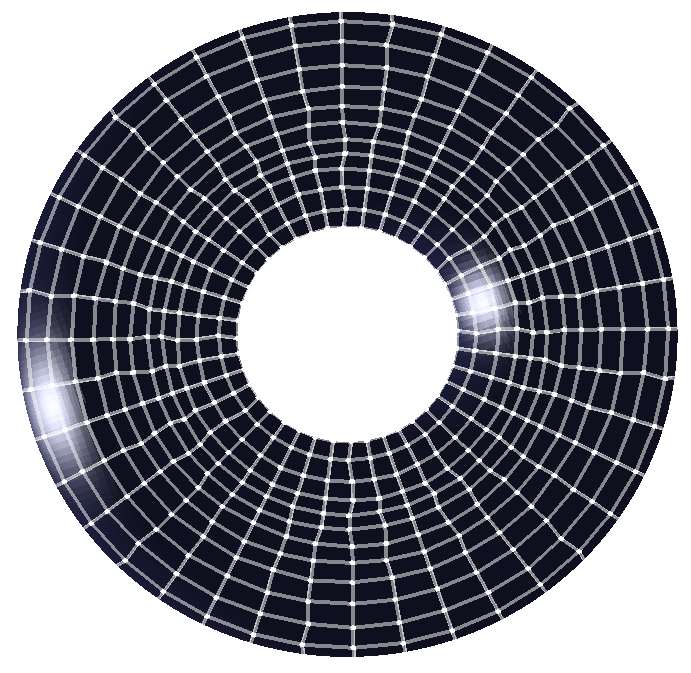}}\label{fig:afterbottom}
	\caption{(a) The mesh on the top of the torus decreases in density as seen in (b) the top of the torus after the mesh is moved via the pushforward map. Meanwhile, (c) the mesh on the bottom of the torus increases in density as seen in (d) the bottom of the torus after the mesh is moved via the pushforward map. The mesh does not tangle after the vertices are transported via the pushforward map.}\label{fig:torus2}
\end{figure}

\section{Conclusion}\label{sec:conclusion}
In this article, we have extended the Optimal Transport problem on a compact surface $\Gamma \subset \mathbb{R}^3$ onto a thin tubular neighborhood $T_{\epsilon}$ with width $\epsilon$. We showed how one can then compute the Optimal Transport mapping $\mathbf{m}$ for the Optimal Transport problem on $\Gamma$ by solving instead for the Optimal Transport mapping $\bar{\mathbf{m}}$ which is the solution to the extended Optimal Transport problem on $T_{\epsilon}$. The key is to extend the density functions and cost function in an appropriate way. The primary benefit of this extension is that the PDE formulation of the Optimal Transport problem on $T_{\epsilon}$ has only Euclidean derivatives. This allows us the flexibility to design a simple discretization that uses a Cartesian grid. We have discretized the extended PDE formulation of the Optimal Transport problem on $T_{\epsilon}$ and shown its ease of implementation and success with two different cost functions on the sphere, on the northern hemisphere of the sphere, where additional boundary conditions had to be implemented, and also on the $2$-torus.

\section*{Acknowledgment}
Tsai's research is supported partially by National Science Foundation Grants DMS-2110895 and DMS-2208504.

{\appendix
\section{Axisymmetric Potential Functions on the unit sphere}\label{sec:exact}


Axisymmetric potential functions $u$ of the Optimal Transport PDE on the sphere can be generated for both the squared geodesic $c(x,y) = \frac{1}{2}d_{\mathbb{S}^2}(x,y)^2$ and logarithmic cost $c(x,y) = -\log(1-x \cdot y)$ by specifying particular axisymmetric source $f$ and target $g$ density functions. For the squared geodesic cost, given that the potential function is axisymmetric, that is, if $u(\theta, \phi) = u(\theta)$, then the Optimal Transport mapping satisfies $T(\theta, \phi) = \left( \theta' , \phi \right)$. Using this, one can compute the determinant of the Jacobian of the Optimal Transport mapping $T$ for the squared geodesic cost as follows:

\begin{equation}
F(\theta):= \left\vert DT \right\vert = \left\vert \frac{\sin \theta '}{\sin \theta} \right\vert \left\vert \frac{d \theta '}{d \theta} \right\vert.
\end{equation}

For the squared geodesic cost function $c(x,y) = \frac{1}{2}d_{\mathbb{S}^2}(x,y)^2$, $\theta ' = \theta + \frac{du}{d \theta}$, and $\theta ' \in [0, \pi]$ and therefore, we derive:

\begin{equation}\label{eq:axisymmetric}
F(\theta) = \frac{\sin \left( \theta + \frac{du}{d \theta} \right)}{\sin \theta} \left\vert \frac{d}{d\theta} \left( \theta + \frac{du}{d \theta} \right) \right\vert= \frac{\sin \left( \theta + \frac{du}{d \theta} \right)}{\sin \theta} \left\vert 1 + \frac{d^2 u}{d \theta^2}  \right\vert.
\end{equation}

In contrast to the derivation in~\cite{scalingskewness}, we are concerned not with the so-called ``monitor function", but instead with producing the density functions $f$ and $g$ that satisfy $f(\theta, \phi) = f(\theta)$ and $g(\theta, \phi) = g(\theta)$. Furthermore, the measure-preserving property implies that $f$ and $g$ must satisfy the following Jacobian equation:

\begin{equation}\label{eq:jac2}
f(\theta) = F(\theta)g \left(\theta + \frac{du}{d \theta} \right).
\end{equation}

To generate simple tests, let us take the target density to be constant, i.e. $g = 1/4\pi$, and we will solve for the source density function $f$ given a known smooth axisymmetric potential function $u$. With a constant target density, the source density function $f$ must satisfy the following:

\begin{equation}
f(\theta) = \frac{F(\theta)}{4 \pi}.
\end{equation}

Now, given the axisymmetric potential function $u = \frac{\cos \theta}{a_0}$, where $a_0 > 1$, we get:


\begin{equation}
f \left(\theta \right) = \frac{\sin \left( \theta - \frac{\sin \theta}{a_0} \right)}{4 \pi \sin \theta} \left( 1 - \frac{\cos \theta}{a_0} \right).
\end{equation}

This integrates to $1$ over $\mathbb{S}^2$, by Equation~\eqref{eq:jac2}. Figure~\ref{fig:exact1} shows what this density function looks like for $a_0 = 3$.

\begin{figure}[htp]
	\includegraphics[width=0.8\textwidth]{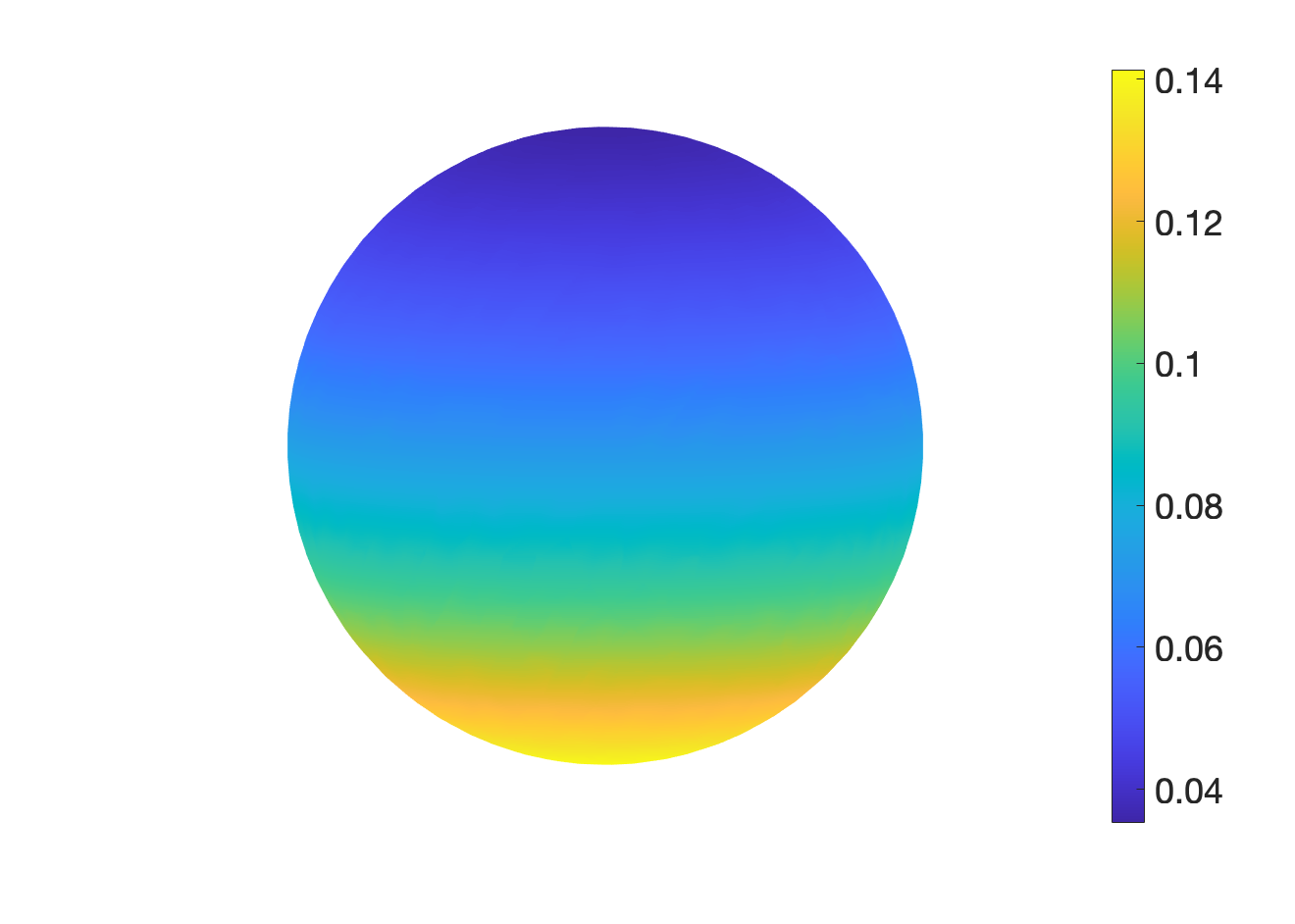}
	\caption{Source density $f$ that leads to a potential function $u(\theta) = \cos(\theta)/3$ for a constant target density $g = 1/4\pi$ for the squared geodesic cost function.}\label{fig:exact1}
\end{figure}

The logarithmic cost is a little trickier. For example, ``flat" potential functions (flat with respect to the standard metric on the sphere) lead to an Optimal Transport mapping $T$ which satisfies $T(\mathbf{x}) = -\mathbf{x}$. For an axisymmetric potential function, such a mapping will satisfy $T(\theta, \phi) = (\theta ', 2 \pi - \phi)$. However, since our densities and potential function will be axisymmetric, we have $T(\theta, 2 \pi - \phi) = (\theta ', \phi)$, so without loss of generality we may just compute $\theta '$ for the Jacobian and not worry about what happens to the aximuthal coordinate $\phi$. For the logarithmic cost, we get $\theta ' = \pi-\theta - \text{sgn} \left( \frac{du}{d\theta} \right) \arccos R \left( \frac{du}{d\theta} \right)$, where $R(\xi) = \frac{\xi^2-1}{\xi^2+1}$. In the axisymmetric case, we either have $\text{sgn} \left( \frac{du}{d\theta} \right) = \pm 1$, since otherwise the mapping would map two or more points to the same point. The expression for the Jacobian is therefore:

\begin{equation}\label{eq:axisymmetric2}
F(\theta) = \frac{\sin \left(\pi-\theta + \text{sgn} \left( \frac{du}{d\theta} \right) \arccos R\left( \frac{du}{d \theta} \right) \right)}{\sin \theta} \left(-1 + \text{sgn}\left( \frac{du}{d \theta} \right) \frac{d}{d \theta} \arccos R\left( \frac{du}{d \theta} \right) \right).
\end{equation}

For the axisymmetric potential function $u = \cos \theta / a_0$, so $\text{sgn} \left( \frac{du}{d\theta} \right) = -1$, and thus, we compute:
\begin{equation}
R \left(\frac{du}{d \theta} \right) = \frac{\sin^2 \theta-a_0}{\sin^2 \theta+a_0},
\end{equation}
and
\begin{equation}\label{eq:axisymmetric2}
f(\theta) = \frac{\sin \left(\pi-\theta - \arccos \left( \frac{\sin^2 \theta - a_0^2}{\sin^2 \theta + a_0^2} \right) \right) }{4 \pi\sin \theta} \left(-1 + \frac{2 a_0\cos \theta}{\sin^2 \theta + a_0^2}   \right).
\end{equation}

Figure~\ref{fig:exact2} shows what this density function looks like for $a_0 = 3$.

\begin{figure}[htp]
	\includegraphics[width=0.8\textwidth]{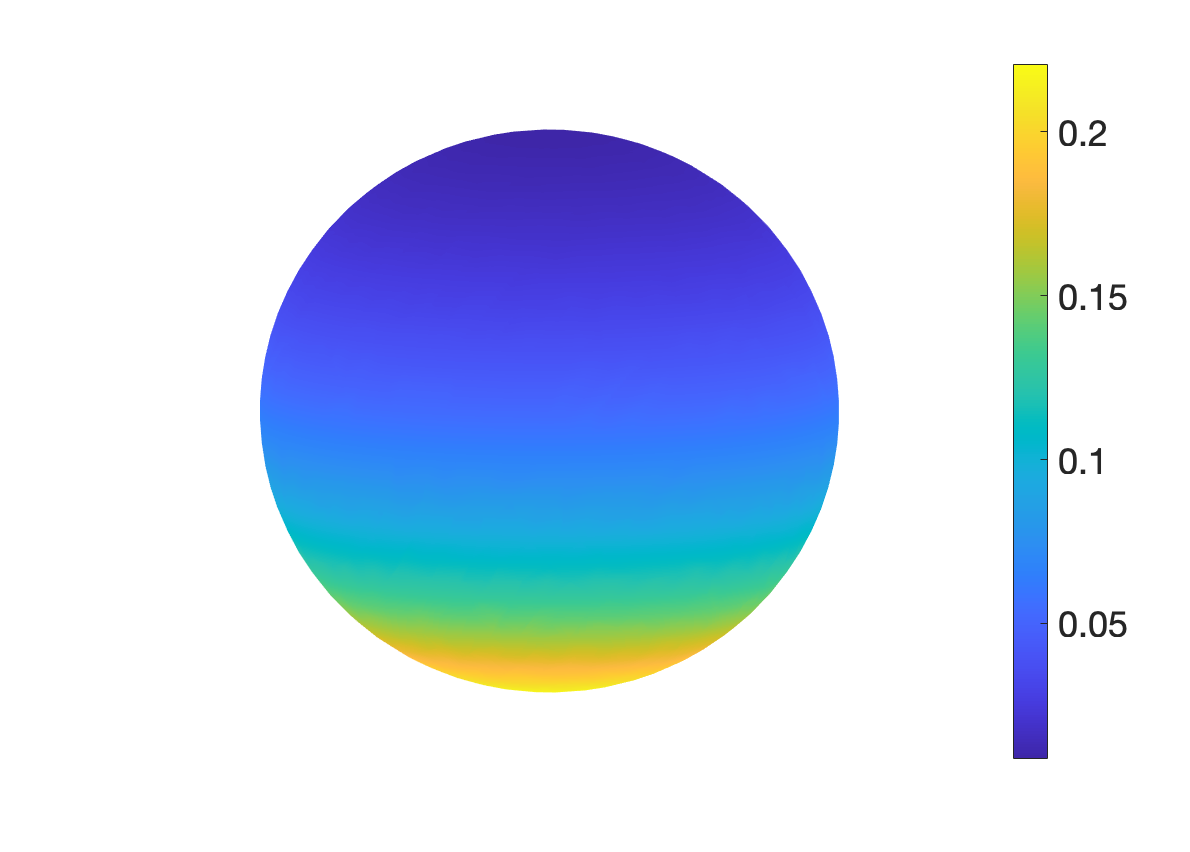}
	\caption{Source density $f$ that leads to a potential function $u(\theta) = \cos(\theta)/3$ for a constant target density $g = 1/4\pi$ for the logarithmic cost function.}\label{fig:exact2}
\end{figure}
}

\bibliographystyle{abbrv}
\bibliography{ShellProject}


\end{document}